\documentclass[11pt]{article}
\usepackage{geometry}                
\usepackage{fullpage}                  
\usepackage{graphicx,xcolor}
\usepackage{amssymb,amsthm,amsmath}
\usepackage{epstopdf}
\usepackage{dsfont}
\usepackage{hyperref}

\newtheorem{prop}{Proposition}
\graphicspath{{figure/}{FIGS/}}

\newcommand{\bE}{\mathbb{E}}
\newcommand{\bP}{\mathbb{P}}
\newcommand{\dd}{\mathrm{d}}
\def\sref#1{section~\ref{#1}}

\def\Fref#1{Figure~\ref{#1}}

\title{Sequential noise-induced escapes for oscillatory network dynamics}%

\author{Jennifer Creaser%
  \thanks{Department of Mathematics, University of Exeter, Exeter, EX4 4QF, UK  and EPSRC Centre for Predictive Modelling in Healthcare, University of Exeter, Exeter, EX4 4QJ, UK (\tt{j.creaser@exeter.ac.uk}, \tt{p.ashwin@exeter.ac.uk}).}%
  \and
  Krasimira Tsaneva-Atanasova%
  \thanks{Department of Mathematics and Living Systems Institute, University of Exeter, Exeter, EX4 4QF, UK and EPSRC Centre for Predictive Modelling in Healthcare, University of Exeter, Exeter, EX4 4QJ, UK ( \tt{k.tsaneva-atanasova@exeter.ac.uk}).}%
  \and
  Peter Ashwin%
  \footnotemark[1]
}

\begin{document}
\maketitle


\begin{abstract}
It is well known that the addition of noise in a multistable system can induce random transitions between stable states. The rate of transition can be characterised in terms of the noise-free system's dynamics and the added noise: for potential systems in the presence of asymptotically low noise the well-known Kramers' escape time gives an expression for the mean escape time.
This paper examines some general properties and examples of transitions between local steady and oscillatory attractors within networks: the transition rates at each node may be affected by the dynamics at other nodes. We use first passage time theory to explain some properties of scalings noted in the literature for an idealised model of initiation of epileptic seizures in small systems of coupled bistable systems with both steady and oscillatory attractors. 
We focus on the case of {\em sequential escapes} where a steady attractor is only marginally stable but all nodes start in this state. As the nodes escape to the oscillatory regime, we assume that the transitions back are very infrequent in comparison. We quantify and characterise the resulting sequences of noise-induced escapes. For weak enough coupling we show that a master equation approach gives a good quantitative understanding of sequential escapes, but for strong coupling this description breaks down.
\end{abstract}



\section{Introduction}

The behaviour of dynamical systems on complex networks has been studied from a variety of viewpoints over the past 30 years, and a variety of tools have been developed to understand cooperative and competitive processes on the network. This is true in many application areas, but particularly in the area of neuroscience.  Oscillatory network models in this area are inspired by oscillatory behaviour present in scales from whole brain regions made of neuronal populations down to single cells, and the networks are fundamental to the organisation of neural systems. For example, a recent review~\cite{AshComNic16} discussed oscillatory models in neuroscience. In the context of pathological states or disease associated brain dynamics, epilepsy serves as a classical example of the importance of oscillatory network dynamics linked to the generation and propagation of epileptiform activity, as it has been discussed in a recent review~\cite{wendling16}. The particular model we study here, given in~\cite{benj12pheno}, is based on the work of~\cite{kalitzin10}, as a network model of epileptic seizures. However, here we are concerned more with the abstract problem of the spreading of noise-induced escapes throughout an oscillatory network. 

In a recent paper \cite{act17domino} we consider {\em sequential noise-induced escapes} for a network of systems where there is escape from a ``shallow" equilibrium attractor to a ``deep" equilibrium attractor at each node. We develop several ideas from that paper to the case where there is bi-stability between steady and oscillatory attractors. As in \cite{act17domino}, each node when uncoupled has two attractors, a stable steady state (that can be destabilised by noise) and a more deeply stable oscillatory attractor (that is more resistant to noise). Starting with the system in the steady attractor, we say it ``escapes" when it crosses a threshold to the basin of the oscillatory attractor. Related work includes, for example, the study of Benayoun et.al.~\cite{ben10ava} who consider a spreading of noise-induced activity in a network of excitatory and inhibitory neurones. More generally sequential transitions between stable/unstable attractors have been implicated in a diverse range of brain functions associated with neuronal timing, coding, integration as well as coordination and coherence \cite{rabinovich08,rabinovich11}.

The time of escape is a random variable that reflects the details of the nonlinear dynamics and the properties of the noise process. In the uncoupled case and for a memoryless escape process, the escapes will be uncorrelated and one can consider independent processes - there will be a random sequence of escapes corresponding to the order in which the nodes escape. More precisely, suppose that a number of bistable dynamical systems each have a ``quiescent'' attractor and an ``active'' attractor, such that in the presence of low amplitude noise there are noise-induced transitions from ``quiescent'' to ``active'' state (that we call ``escape" of the system) but not vice versa. Coupling of such systems can promote (or decrease) escape of others on the network.  There may be critical values of the coupling, as identified in \cite{act17domino,BFG2007a,BFG2007b}, at which the qualitative nature of the escape changes associated with bifurcations on basin boundaries of the attractors where typical transitions occur. In this sense one can see the sequences of escapes, and their relative timings and probabilities, as emergent properties of the network. 

Throughout this paper we link our work to the Eyring-Kramers escape time~\cite{eyring35activ,kramers40brown} between potential minima. In the classical one-dimensional case the expected escape time $T$ from a local (quadratic) minimum $x$ of a potential $V,$ over the unique local (quadratic) maximum $z$ is given by
\begin{equation}
T \simeq \frac{2\pi}{\sqrt{V''(x)|V''(z)|}}\text{e}^{[V(z) - V(x)]/\varepsilon}.
\label{eq:kramers1D}
\end{equation}
We also make use of the multidimensional analogue of \eqref{eq:kramers1D} that assumes minima are separated by a unique saddle~\cite{eyring35activ,kramers40brown}. The first proof of the multidimensional Erying-Kramers' Law, including a definition of $\simeq$, was given in \cite{bovier04meta} using, among other things, potential theory. We also use an analysis based on \cite{berg08EK} that gives generalised Kramers' scalings near a pitchfork bifurcation on the basin boundary of the local minimum from which escape occurs.

In this paper we examine in Section~\ref{sec:1node} the behaviour of a single phenomenological node considered in \cite{benj12pheno} that has bistability between steady and oscillatory attractors. We study in detail the noise-induced escapes from steady to oscillatory attractors and characterise a condition such that the escape from the steady attractor occurs more frequently than escape from the oscillatory attractor in the limit of low noise. We use standard mean escape-time theory to obtain closed form expressions for the mean escape time from steady state, and verify upper and lower bounds in terms of the problem parameters, thus improving the asymptotic estimates presented in \cite{benj12pheno}. 

In Section~\ref{sec:2node} we consider two coupled identical bistable nodes of the form discussed in Section~\ref{sec:1node}. For the cases of a pair of bidirectional, unidirectional and uncoupled nodes we are able to use potential theory analysis of the stochastically forced coupled system to explain the scalings of mean escape times as a function of coupling strength, as observed numerically by \cite{benj12pheno}. In particular for strong bidirectional coupling we find (somewhat counterintuitively) that the mean escape time for one node is greatly increased by the coupling, but the mean escape time of the second node is greatly reduced: this completes the work presented in \cite{benj12pheno} that concerns only the first escape of one of the nodes. As previously discussed in a symmetric context \cite{BFG2007a,BFG2007b} this behaviour is due to the presence of bifurcations in the basin boundary of the steady attractor that correspond to synchronisation, though here the phase dynamics of the coupled oscillations adds an extra complication.

We extend this to some simple networks in Section~\ref{sec:master}. In this context we introduce a master equation approach to the problem of sequential escapes in such a bistable network. For weakly bidirectionally coupled networks of identical bistable nodes, this approach gives a good abstract model representation for the sequential escape process as long as the coupling is sufficiently weak. For the system considered, this description breaks down via a bifurcation process that occurs when the coupling strength reaches a critical value.  Finally, we briefly discuss some open problems and extensions of this work in Section~\ref{sec:discuss}.

\section{Single node escape times}
\label{sec:1node}

The phenomenological network model for seizure onset studied in~\cite{benj12pheno} considers idealised nodes that can be stable in either steady or oscillatory behaviours.
This is probably the simplest planar system that gives coexistence of steady and oscillatory attractors. In \cite{benj12pheno,kalitzin10} the motivation was to regard this as a representation of the brain activity measured by electroencephalogram (EEG) that may be in a healthy (non-oscillating) or seizure (oscillating) state.
 We consider the complex valued noise-driven system 
\begin{equation}
	\dd z(t) = f(z)\dd t +\alpha \dd\, W(t)
	\label{eq:ben1}
\end{equation}
where $W(t)=u+i v$ is a complex Wiener process ($u$ and $v$ are real independent Wiener processes) with noise amplitude $\alpha>0$, and
\begin{equation}
	f(z) = (\nu +i\omega)z +2z|z|^2 - z|z|^4
	\label{eq:benf}
\end{equation}
can be thought of as noise-driven truncated normal form of a Bautin bifurcation:
\begin{equation}
	\dot{z}=f(z).
	\label{eq:bennonoise}
\end{equation} 
For $\nu<0$ the only attractor of (\ref{eq:bennonoise}) is a stable periodic orbit surrounding an unstable equilibrium at the origin. This equilibrium becomes stable in a subcritical Hopf bifurcation at $\nu=0$ and in the regime $0<\nu<1$ the system exhibits bistability with an attracting fixed point and an attracting limit cycle separated by an unstable limit cycle: the stable and unstable periodic orbits meet each other in a saddle node bifurcation at  $\nu=1$. The parameter $\omega$ controls the frequency of the oscillations and here we fix $\omega=20$ as in~\cite{benj12pheno}. Figure~\ref{fig:1nodedyn} summarises the dynamics of \eqref{eq:ben1}, where one realisation of \eqref{eq:ben1} for $\alpha=0.2$ and $\nu=0.5$ is shown in the phase space of \eqref{eq:benf} in panel (a). The time series of the realisation is shown in panel (b). Panel (c) summarises the bifurcation diagram of~\eqref{eq:ben1}.

\begin{figure}[t]
\includegraphics[width=\textwidth]{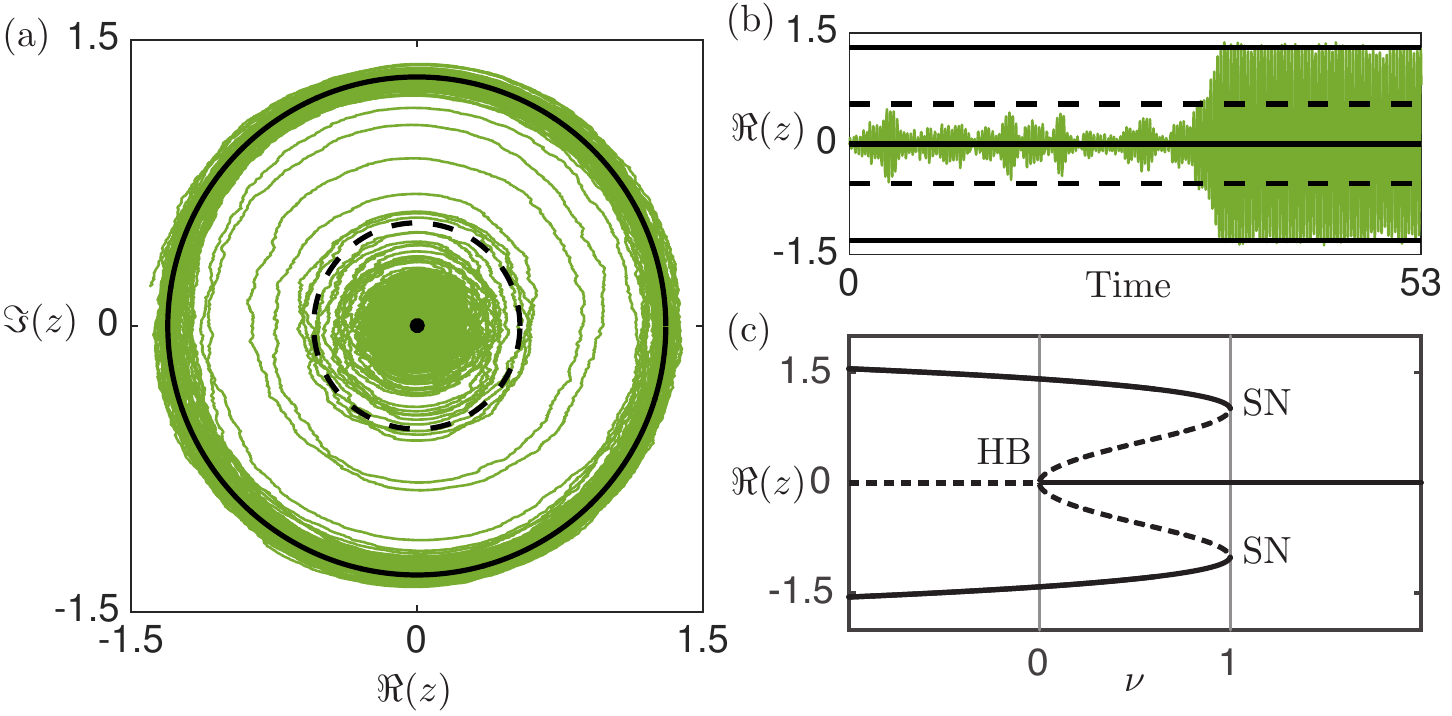}
\caption{The single node noise-driven dynamics of \eqref{eq:ben1}. 
Panels (a) and (b) show one realisation in green for $\nu=0.5$, $\alpha=0.2$ and $\omega=20$  in the phase space of \eqref{eq:benf} containing a stable equilibrium at the origin, an unstable limit cycle (dotted line) and stable limit cycle (solid line).
The bifurcation diagram of \eqref{eq:benf} is shown in panel (c), with the the Hopf  (HB) and saddle-node (SN) bifurcations marked.
The dashed lines in panel (c) indicate the unstable equilibrium at the origin for $\nu<0$ and the unstable limit cycle for $0<\nu<1$. We show in Section~2 that for $\nu<3/4$ the limit cycle is more stable than the equilibrium, in the sense that the potential for the radial dynamics is lower.
}
\label{fig:1nodedyn}
\end{figure}

In the presence of noise of amplitude $\alpha>0$, both steady and oscillatory attractors of (\ref{eq:ben1})  show stochastic fluctuations and there are occasional transitions between these two metastable states, driven by occasional large fluctuations in the noise. For low noise amplitude (\ref{eq:ben1}) shows similar behaviour to the underlying deterministic system, whereas for large noise the dynamics are dominated by large stochastic fluctuations. 

The dynamics' realisation shown in Figure~\ref{fig:1nodedyn} is computed in {\sc{Matlab}} using the Heun method for stochastic differential equations~\cite{kloeden03num} with the initial condition at the origin and step size $h=10^{-5}$. The realisation trajectory spends some time near the origin but the stochastic fluctuations eventually drive it past the basin boundary represented by the unstable limit cycle and it is then attracted to the stable limit cycle. The time series shows this transition from small, noise dominated fluctuations, to an oscillatory regime.

In the presence of noise, we define, analogously to \cite{act17domino}, the escape time of the node $\tau$ to be when the realisation trajectory switches from being close to the origin (quiescent) to being close to the stable limit cycle (active). More precisely, if the noise-free system has stable limit cycles at $|z|=0$ and $|z|=R_{\max}$ separated by an unstable limit cycle at $0<|z|=R_c<R_{\max}$, then the escape time for a given threshold $\xi\in(R_c,R_{\max})$  is
$$
\tau = \inf\{t>0:|z(t)|>\xi~\mbox{given}~z(0)=0\}.
$$
Note that $\tau$ is a random variable that depends on the noise realisation and reflects the influence of the noise on the nonlinear dynamics. For small enough noise the escape time has a cumulative distribution $Q(t) = \bP \{ \tau<t\}$ with an exponential tail \cite{berg13kramers} and the mean escape time $T$ from the steady to the oscillatory attractor is
\begin{equation}
T=\bE(\tau) = \int_{t=0}^{\infty} t \frac{d}{dt}Q(t)\,dt= \int_{t=0}^{\infty} [1-Q(t)]\,dt.
\label{eq:meanesc}
\end{equation}
This is what we aim to quantify in the following section.

\subsection{Mean escape times for a single node}
\label{sec:1nodemean}

To determine the mean escape time, we transform~\eqref{eq:ben1} into polar coordinates given by $z(t)=R(t)\exp[\imath \theta(t)]$ with $R(t)\geq 0$ and $\theta(t)$ considered modulo $2\pi$. This gives
\begin{align}
	\dd R & = \Bigg{[} - \nu R + 2R^3 - R^5 + \frac{\alpha^2}{2R} \Bigg{]} \dd t +\alpha \dd W_R \label{eq:benr}\\ 
	\dd \theta &= \omega \dd t + \frac{\alpha}{R} \dd W_\theta
	\label{eq:benth}
\end{align}
where $W_R$ and $W_\theta$ are independent standard Wiener processes. The $\alpha^2/R$ terms arise from It\^{o}'s Lemma; see for example \cite{daffer98,gard83hand}. As the $R$ equation is independent of $\theta$ we can consider the escape problem to oscillatory behaviour as a one-dimensional potential (well/energy) problem for $R(t)$
\begin{equation*}
{\dd R} = -\frac{\partial V}{\partial R} \dd t+\alpha \dd W_R,
\end{equation*}
where the potential function, $V$, is given by 
\begin{equation}
V:=\frac{\nu R^2}{2} - \frac{R^4}{2} + \frac{R^6}{6}-\frac{\alpha^2}{2} \ln{R}.
\label{eq:Vnontrunc}
\end{equation}

The maxima and minima of $V$ correspond to the equilibrium and limit cycles of the full system. Note that the potential depends on $\alpha$; for $\alpha=0.3$ and $\nu=0.5$, $\frac{\partial V}{\partial r}=0$ at exactly two points which are the potential wells corresponding to the stable limit cycle. However, as the noise amplitude increases the potential barrier decreases and disappears and so the escape time of the node decreases. More precisely, one can determine the bifurcation behaviour of the ODE
\begin{equation}
-\frac{\dd R}{\dd t} = \frac{\dd V}{\dd R}  =- \frac{\alpha^2}{2R} + \nu R - 2R^3 + R^5
\label{eq:delalpha}
\end{equation}
in the region $R>0$ with $V$ as in \eqref{eq:Vnontrunc} as a function of $\alpha>0$ and $\nu>0$. One can verify (eliminating $R$ from the conditions $V'(R)=V''(R)=0$) that there are saddle node bifurcations for this system at \begin{equation}
\nu^3-\nu^2-\frac{9}{2}\nu \alpha^2+\frac{27}{16}\alpha^4+4\alpha^2=0
\label{eq:saddlenodecurve}
\end{equation}
that has a cusp point (where $V'(R)=V''(R)=V'''(R)=0$) at $(\nu,\alpha)=(4/3,4\sqrt{3}/9)$.
Hence one can verify the existence of three equilibria for $R>0$ in a region near $\alpha=\nu=0$ bounded by $0<\alpha<\nu/2+O(\nu^2)$.

Within the bounded region of $(\alpha,\nu)$ given by (\ref{eq:saddlenodecurve}) there are three distinct equilibria of \eqref{eq:delalpha} at parameter-dependent locations we denote
$$
0<R_{\min}<R_{c}<R_{\max}.
$$
The $R_{\min}$ and $R_{\max}$ are attractors corresponding to minima of $V$ while $R_{c}$ is a local maximum (unstable) that forms a {\em gate} (boundary) between the basins of the two attractors in the terminology of \cite{berg08EK}. Note for $\alpha=0$ and $0<\nu<1$ the three equilibria are at $R_{\min}=0$, $R_{c}=R_c^0:=\sqrt{1-\sqrt{1-\nu}}$ and $R_{\max}=\sqrt{1+\sqrt{1-\nu}}$.

\begin{figure}
\centering
\includegraphics[width=\textwidth]{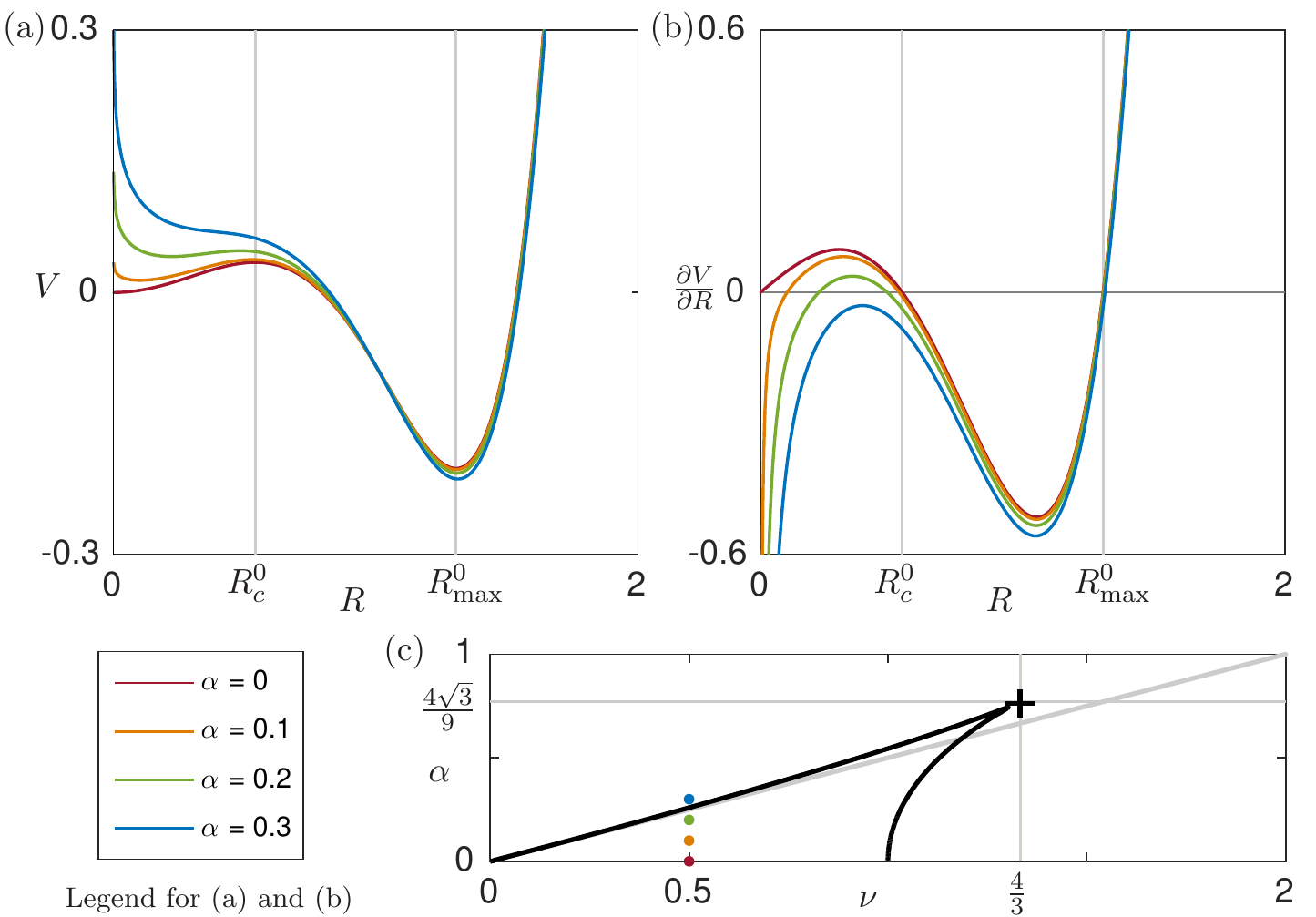}
\caption{The effect of varying $\nu$ and $\alpha$ on the potential function $V$ for $R>0.$ The curves $V$ (a) and $\frac{\partial V}{\partial R}$ (b) for $\nu=0.5$ and different values of $\alpha$. In panel (a) the deeper well corresponds to the stable limit cycle and the peak to the unstable limit cycle; note the asymptote at $R=0$ for $\alpha>0$. Each extrema in panel (a) corresponds to $\frac{\partial V}{\partial R}=0$ in panel (b). The saddle node bifurcation curve is shown in the $(\nu,\alpha)$-plane with cusp point marked $+$ in panel (c).  The parameter points marked as coloured dots in (c) for $\nu=0.5$ correspond to the curves in panels (a) and (b); see also the legend. The line $\nu=2\alpha$ is also marked and for $\nu<1$ this is very close to the bifurcation curve. For $\nu>2\alpha$ there are is a well of $V$ near the origin and $\frac{\partial V}{\partial R}=0$ for three values of $R>0$, for $\nu<2\alpha$ there is no well near the origin and $\frac{\partial V}{\partial R}=0$ only once in $R>0$.
}
\label{fig:1nodepot}
\end{figure}

Figure~\ref{fig:1nodepot} shows the potential $V$ and its derivative $\frac{\partial V}{\partial R}$ for different values of $\alpha$ along with the saddle-node bifurcation curve in the $(\nu,\alpha)$-plane. Due to the symmetry of $V$, a reflection at $R=0$, along which the $\ln R$ term creates an asymptote for $\alpha>0$, we plot $V$ and $\frac{\partial V}{\partial R}$ for $R>0$. The zeros of $\frac{\partial V}{\partial R}$ in panel (b) correspond to the extrema of $V$ and the equilibria $R_c^0$ and $R_{\text max}^0$ are marked. Panel (c) shows the cusp point  of the saddle-node bifurcation curve in the $(\nu,\alpha)$-plane with the line $\nu=\alpha/2$. The dots mark the parameter values of the curves in panels (a) and (b); note that $\nu=0.5, \alpha=0.3$ is not within the bounded region and the corresponding curve in panel (b) only has one zero. 

Figure~\ref{fig:1nodepot} shows that for $\alpha>0$ we have that whenever $R_c$ exists, it satisfies $R_{c}<R_{c}^0$. We choose a threshold $\xi$ such that 
$$
R_c<\xi<R_{\max}.
$$
Although the leading order escape times are independent of $\xi$, for comparability to \cite{benj12pheno} we take $\xi=R_c^0$ in this section. In later sections we take a fixed value of $\xi=0.5$.

The mean escape time $T(\nu,\alpha)$ from (\ref{eq:meanesc}) can be found by considering solutions $R(t)$ of the SDE (\ref{eq:benr}) and defining the mean first escape time
$$
W_\xi(R_0) := \bE ( \inf\{t>0~:~R(t)\geq \xi,~\mbox{given}~R(0)=R_0\}).
$$
This mean escape time $W_{\xi}(R)$ satisfies a Poisson problem
\begin{equation}
	\frac{\alpha^2}{2} \frac{\dd ^2}{\dd R^2} W_{\xi}(R) - V'(R)\frac{\dd }{\dd R} W_{\xi}(R) =-1, \ \ \lim_{R\rightarrow 0+} |W_{\xi}(R)|<\infty, \ \ W_{\xi}(a)=0.
	\label{eq:pois}
\end{equation}
If we define $u(R)=\frac{\dd  W_{\xi}}{\dd R}$ then \eqref{eq:pois} can be simplified using an integrating factor $\exp(\frac{-2V}{\alpha^2})$.
The boundary conditions imply $W_{\xi}(R) \to 0$ and 
$ \frac{{\rm{d}W_{\xi}}}{\dd R}\exp (\frac{-2V}{\alpha^2}) \to 0$ as $R \to 0$. 
Integrating and applying the boundary conditions
\begin{equation}
	W_{\xi}(R)= \frac{2}{\alpha^2} \int_{x=R}^{\xi} \int^x_{y=0} \exp\bigg[\frac{2(V(x)-V(y))}{\alpha^2}\bigg] \dd y\, \dd x.
\label{eq:WarVs}
\end{equation}
Substituting the expression for $V$ from \eqref{eq:Vnontrunc} gives  
\begin{align*}
	W_{\xi}(R)
	&= \frac{2}{\alpha^2} \int_{x=R}^{\xi} \int^x_{y=0} \exp\frac{2}{\alpha^2}\bigg[ \frac{\alpha^2}{2}( \ln{y}- \ln{x} )  +  \frac{\nu( x^2-y^2)}{2}+ \frac{y^4-x^4}{2}+\frac{x^6-y^6}{6}  \bigg] \dd y\, \dd x.
\end{align*}
Therefore, as $T(\nu, \alpha) = W_{\xi}(0)$ we have 
\begin{equation}
	T(\nu, \alpha) = \frac{2}{\alpha^2} \int_{x=0}^{\xi} \int^x_{y=0} \frac{y}{x} \exp\bigg(\frac{\nu(x^2-y^2)+(y^4-x^4)+(x^6-y^6)/3}{\alpha^2}\bigg) \dd y\, \dd x.
	\label{eq:meanT}
\end{equation}
Kramers' formula \cite{berg13kramers} uses Laplace's method to give an asymptotic expression for \eqref{eq:WarVs}:
\begin{equation}
	T_{K}(\nu,\alpha) = \frac{2\pi}{\sqrt{|V''(R_c)|V''(R_{\min})}} \exp \left[\frac{2(V(R_c)-V(R_{\min})}{\alpha^2} \right]
\end{equation}
as $\alpha\rightarrow 0$. Moreover, one can obtain upper and lower bounds on (\ref{eq:meanT}) that are valid for general $0<\nu<1$ and $\alpha>0$ (cf. \cite{ashwin16quant}). We write 
\begin{equation}
p=x^2-y^2,~~q=x^2+y^2,~~\Rightarrow~~ \frac{\partial(p,q)}{\partial(x,y)}= \left|\begin{array}{rr} 2x & -2y \\ 2x & 2y \end{array}\right|=8 xy.
\end{equation}
The triangle defined by $(x,y)$ such that $0<y<x<\xi$ transforms to $0<q<2\xi^2$, $0<p<\min(q,2\xi^2-q)$, so we have
\begin{equation}
T(\nu, \alpha) = \frac{1}{\alpha^2} \int_{q=0}^{2\xi^2} \int_{p=0}^{\min(q,2\xi^2-q)} \frac{1}{2(p+q)} \exp\bigg(\frac{p(\nu-q+p^2/12+q^2/4)}{\alpha^2}\bigg) \text{d}p\, \text{d}q.
\label{eq:T}
\end{equation}

Noting that $0<p<q$ in the region of integration implies $q<p+q<2q$, in addition noting  $q^2/3>p^2/12+q^2/4>q^2/4$ in this region we obtain the following estimates for the integrand of (\ref{eq:T})
\begin{align*}
\frac{1}{2q}\exp\bigg(\frac{p(\nu-q+q^2/3)}{\alpha^2}\bigg) &>\frac{1}{2(p+q)}\exp\bigg(\frac{p(\nu-q+p^2/12+q^2/4)}{\alpha^2}\bigg)\\
\frac{1}{4q}\exp\bigg(\frac{p(\nu-q+q^2/4)}{\alpha^2}\bigg)&< \frac{1}{2(p+q)}\exp\bigg(\frac{p(\nu-q+p^2/12+q^2/4)}{\alpha^2}\bigg)
\end{align*}
Hence, we have lower and upper bounds $T_l(\nu,\alpha)<T(\nu,\alpha)<T_u(\nu,\alpha)$
given by:
\begin{align}
T_l(\nu, \alpha) & = \int_{q=0}^{\xi^2} \int_{p=0}^{q} \frac{1}{4q\alpha^2} \exp\bigg(\frac{p(\nu-q+q^2/4)}{\alpha^2}\bigg) \dd p\, \dd q,\\
T_u(\nu, \alpha) &= \int_{q=0}^{2\xi^2} \int_{p=0}^{q} \frac{1}{2q\alpha^2} \exp\bigg(\frac{p(\nu-q+q^2/3)}{\alpha^2}\bigg) \dd p\, \dd q.
\end{align}
The inner integrals can be evaluated to give
\begin{align}
T_l(\nu, \alpha) & = \int_{q=0}^{\xi^2} \frac{e^{\frac{q(\nu-q+q^2/4)}{\alpha^2}}-1}{4q(\nu-q+q^2/4)}  \dd q, \label{eq:Tl} \\
T_u(\nu, \alpha) &= \int_{q=0}^{2\xi^2} \frac{e^{\frac{q(\nu-q+q^2/3)}{\alpha^2}}-1}{2q(\nu-q+q^2/3)} \dd q.\label{eq:Tu}
\end{align}

Moreover, note that Laplace's method can be used to get an asymptotic estimate for these bounds in the case of fixed $\nu$ and $\alpha\rightarrow 0$. Define
\begin{align*}
f_l(q)=(4q(\nu-q+q^2/4))^{-1},~~&f_u(q)=(2q(\nu-q+q^2/3))^{-1},\\ g_l(q)=q(-\nu+q-q^2/4),~~&g_u(q)=q(-\nu+q-q^2/3).
\end{align*}
One can verify that for a fixed $\nu \in (0,1)$, the functions $g_l(q)$ and $g_u(q)$ have unique minima on $[0,\xi^2]$ and $[0,2\xi^2]$ at $c_l=(4-2\sqrt{4-3\nu})/3$ and $c_u=1-\sqrt{1-\nu}$, respectively. It is also easy to verify that $f_l(c_l) \neq 0$ and $f_u(c_u) \neq 0$. Thus, to leading order as $\alpha \to 0$. 
\begin{align*}
T_l(\nu, \alpha) &\sim f_l(c_l)\sqrt{\frac{2\pi\alpha^2}{|g_l^{\prime \prime}(c_l)|}}\exp \left[-\frac{g_l(c_l)}{\alpha^2}\right],  \\
T_u(\nu, \alpha) &\sim f_u(c_u)\sqrt{\frac{2\pi\alpha^2}{|g_u^{\prime \prime}(c_u)|}}\exp\left[-\frac{g_u(c_u)}{\alpha^2}\right].
\end{align*}

The functions derived here show the direction connection between our analytic results for the escape time of one node given by~\eqref{eq:benr} and the classic one-dimensional Kramers' escape time~\eqref{eq:kramers1D} as well as the dependence on the excitability $\nu$ and the noise amplitude $\alpha$.

\Fref{fig:1nodeaTnT} shows the integral $T(\nu,\alpha)$ numerically estimated using {\sc Maple} from \eqref{eq:T} plotted against $\alpha$ for different values of $\nu$ in panel (a), and plotted against $\nu$ for different values of $\alpha$ in panel (b); compare with~\cite{benj12pheno} Figure~5. Panel (a) shows the lower and upper bound $T_l(\nu,\alpha)$ and $T_u(\nu,\alpha)$ also numerically estimated using {\sc Maple} from~\eqref{eq:Tl} and~\eqref{eq:Tu} respectively. The approximation of $T(\nu,\alpha)$ from~\cite{benj12pheno} is shown for comparison for each $\nu$ value.  In panel (b) a cross is marked on the curve $\alpha=0.05$ at $\nu=0.2$, we use these values in subsequent sections.   The panels also show points that are numerical approximations of the mean escape time, computed in {\sc Matlab} using the Heun method. For each point, two hundred realisations of~\eqref{eq:ben1} were computed with step size $h=10^{-2}$. As the radial dependence \eqref{eq:benr} is independent of $\omega$, we fix $\omega=0$ in our computations.  The figure uses threshold $\xi=R_c^0$, as in \cite{benj12pheno}, which corresponds to the amplitude of the unstable periodic orbit of \eqref{eq:benf} for $\alpha=0$.  Comparing to the approximation of $T$ from~\cite{benj12pheno} we find reliable bounds and good agreement with the numerics over a large range of $(\alpha,\nu)$. 

\begin{figure}[ht!]
	\centering
	\includegraphics[width=\textwidth]{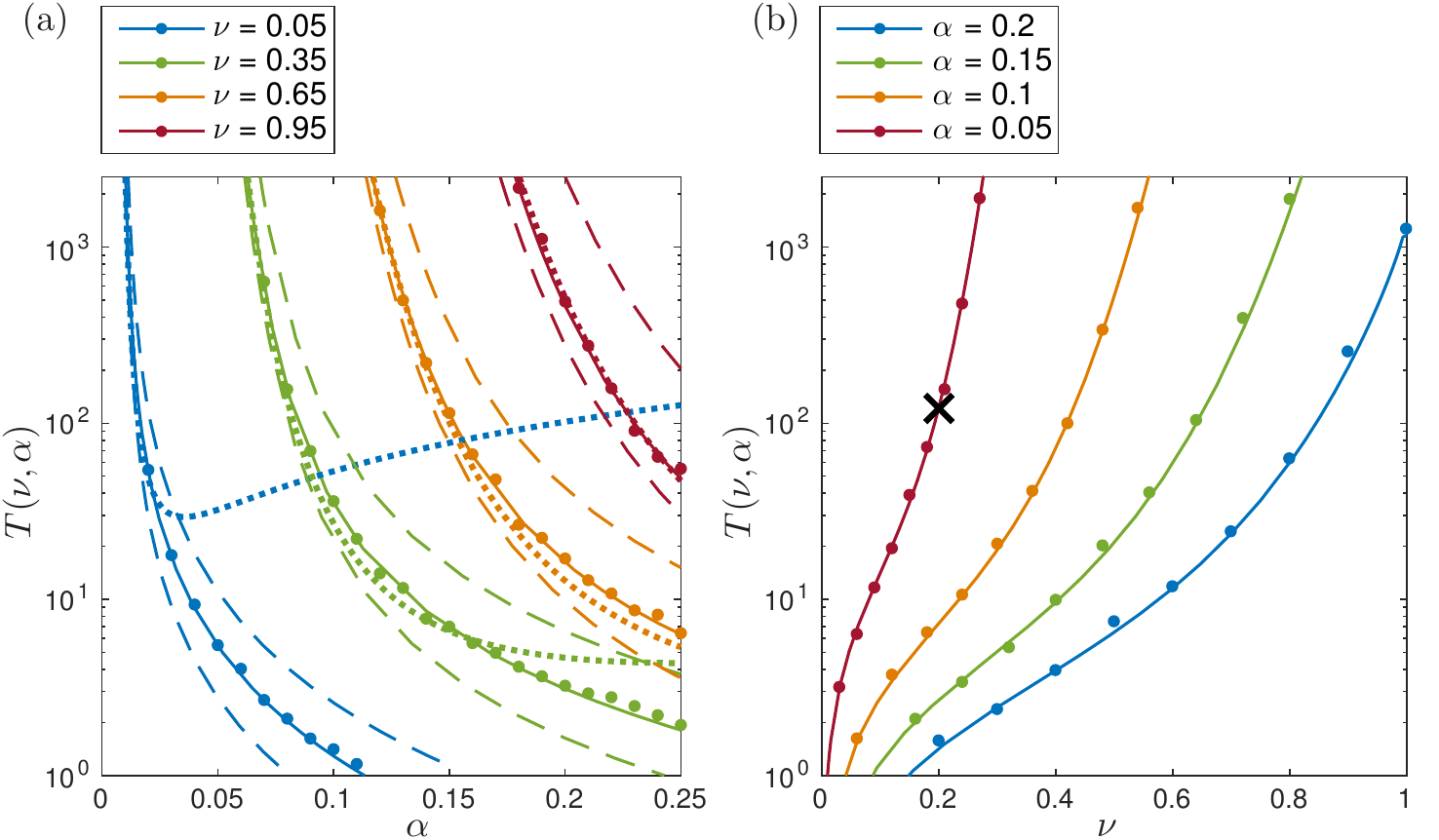}
	\caption{Numerical approximations of the mean escape time $T(\nu,\alpha)$ from $R=0$ to threshold $\xi=R_c^0$ (solid) from (\ref{eq:T}), along with mean escape times from numerical simulations of escape time (dots). Panel (a) shows the upper bound $T_u(\nu,\alpha)$ and lower bound $T_l(\nu,\alpha)$ computed from (\ref{eq:Tl}) and (\ref{eq:Tu}) respectively (dashed). The dotted curves show the estimate of $T(\nu,\alpha)$ from Benjamin {\em et al} {\cite[Figure 5]{benj12pheno}}. The cross in panel (b) is on the curve $\alpha=0.05$ at $\nu=0.2$ and corresponds to $T(0.2,0.05)=121.64$.}
	\label{fig:1nodeaTnT}
\end{figure}

\section{Sequential escape times for coupled bistable nodes}
\label{sec:2node}

We now consider $N$ identical bistable nodes of the type \eqref{eq:ben1} discussed in Section~\ref{sec:1node}, coupled linearly as in~\cite{benj12pheno}:
\begin{equation}
	\dd z_i(t) =\left[ f(z_i) + \beta \sum_{j\neq i} A_{ji}(z_j - z_i) \right]\dd t +\alpha\, \dd  W_i (t),
	\label{eq:neteq}
\end{equation}
for $i=1,...,N$, where $f$ is defined in \eqref{eq:benf} and depends on $\nu$. This generalises the setting of \cite{act17domino} to a case of bistable nodes where one of the attractors is periodic. For this network, $A_{ji}\in\{0,1\}$ is the adjacency matrix and $\beta\geq 0$ the coupling strength: we assume that $A_{ii}=0$ for all $i$.

We fix $0<\nu<1$ and $\nu>2\alpha$ to ensure that each individual node is in the bistable regime with an attracting equilibria for the radial dynamics (\ref{eq:benr}) at $R_{\min}$ and $R_{\max}$. For sequential escapes, we will assume the parameter regime is such that the rate of return from $R_{\max}$ to $R_{\min}$ is very small relative to the rate of escape from $R_{\min}$ to $R_{\max}$. This can be quantified in terms of the potential (\ref{eq:Vnontrunc}): for $\alpha=0$ and $0<\nu<1$ we have
$$
V(R_{\min})=0,~~V(R_{\max})=\frac{\nu}{2}-\frac{1}{3}\left(1+(1-\nu)^{\frac{3}{2}}\right).
$$
One can verify for this case that there are two attractors and
\begin{equation}
	V(R_{\max})<V(R_{\min}).
	\label{eq:marginalzero}
\end{equation}
if and only if $0<\nu<3/4$. Moreover, for $0<\nu<3/4$ fixed and increasing $\alpha$, (\ref{eq:marginalzero}) is maintained as long as there are still three roots: the $-\alpha^2/2 \ln R$ dependence means that $V(R_{\min})$  increases while $V(R_{\max})$ decreases with $\alpha$. 

If all nodes start in the quiescent state it is a natural question to ask how the coupling affects the sequence of escape times of the nodes in the network \cite{act17domino}. 
We discuss the general set-up in the next section and then focus on the example of two  coupled nodes.

\subsection{Statistics of sequential escapes}
\label{sec:seqstats}

We fix a threshold $\xi>0$ for all nodes and consider the first escape time for the $i$th node 
$$
\tau^{(i)} := \inf\{t>0~:~ |z_i(t)|\geq \xi~\mbox{ given }~ z_i(0)=0\}
$$
from the quiescent state, assuming that all nodes start at $z_k=0$ ($k=1,\ldots,N$) at time $t=0$. The distribution of the random variable $\tau^{(i)}$ is affected by the noise realization and the chosen threshold $\xi$, as well as the behaviour of other nodes in the network via the coupling structure $A_{ij}$ and strength $\beta$. We choose $\xi$ such that the region $|z_i|<\xi$ contains the whole basin of attraction of the trivial solution $z=0$ in the limit $\alpha\rightarrow 0$. Note that the coupling deforms the basin of attraction and so it may be necessary to choose $\xi$ somewhat greater than $R_c$, depending on $\beta$.

For a fixed threshold $\xi$ and initial condition, the independence of the noise paths means that, with probability one, no two escapes will be at precisely the same time. Therefore, there will be a sequential ordering of nodes corresponding to the order in which they escape. Using the notation of~\cite{act17domino}, there is a permutation $s(k)$ of $\{1,\ldots,N\}$  such that 
\begin{equation}
0<\tau^{(s(1))}<\tau^{(s(2))}<\cdots<\tau^{(s(N))}
\label{eq:defines}
\end{equation}
where the times $\tau^{(s(k))}$ and the permutation $s(k)$ are random variables that depend on the realization of the noise. We also define
$$
\tau^{i}:=\tau^{(s(i))}
$$
which can be thought of as the time until the $i$th escape, and we write $\tau^{0}=0$. For any integers $0\leq \ell<k\leq N$
\begin{equation}
\tau_N^{k|\ell}:=\tau^{k}-\tau^{\ell}
\label{eq:fpt}
\end{equation}
is the first passage time between the $\ell$th and $k$th escapes. Although \cite{act17domino} considers both the timing and sequence of escapes, in this paper we concentrate primarily on $\tau^{k|\ell}_N$.

Sequential escape can be understood in terms of this set of distributions which are governed by distributions with exponential tails and therefore by Kramers-type rates. In these cases, the essential information about the sequential escapes is given by the mean first passage time between escapes $\ell$ and $k$ that is the expectation of $\tau_N^{k|\ell}$, i.e.
\begin{equation}
T_N^{k|\ell}: =  \bE\left(\tau_N^{k|\ell}\right)=\int_{t=0}^{\infty}  \left[1-Q_N^{k|\ell}(t)\right]\, \dd t,
\label{eq:mfpt}
\end{equation}
where
\begin{equation}
Q_N^{k|l}(t) = \mathbb{P}(\tau_N^{k|l} \leq t)
\label{eq:cdfpt}
\end{equation}
is the cumulative distribution of first passage times from $\ell$ to $k$.
Note that if $k>\ell>n$ then $\tau_N^{k|n}=\tau_N^{k|\ell}+\tau_N^{\ell|n}$, and so taking expectations we have
\begin{equation}
T_N^{k|n}=T_N^{k|\ell}+T_N^{\ell|n}.
\end{equation}
Section~\ref{sec:master} returns to this general case in more detail, while for the rest of this section we consider specific examples of sequential escapes for \eqref{eq:neteq} in the case $N=2$.

We consider three coupling scenarios for \eqref{eq:neteq} with $N=2$: disconnected ($\beta=0$ or $A_{12}=A_{21}=0$); unidirectional ($A_{12}=1$ but $A_{21}=0$); and bidirectional ($A_{12}=A_{21}=1$). The study~\cite{benj12pheno} investigates the influence of $\beta$ on the mean first passage time such that the first node has escaped, i.e. $T_2^{1|0}$, but clearly $T_2^{2|1}$ is also of interest. The paper \cite[Figure~6(b)]{benj12pheno} shows a number of limiting behaviours that we aim to explain here using the potential function for the coupled system.  Here, we focus mainly on the case of two bidirectionally coupled nodes with brief comparison to the unidirectionally coupled and uncoupled cases.

\subsection{Two bidirectionally coupled nodes}
\label{sec:2binet}

 Writing system~\eqref{eq:neteq} for $N=2$ with bidirectional coupling ($A_{12}=A_{21}=1$) in polar coordinates $z_i(t) = R_i(t)\exp[\imath \theta_i(t)]$ we have
\begin{align}
	\dd R_i &= \bigg[-(\nu+\beta)R_i + 2R_i^3 - R_i^5 + \beta R_k\cos(\phi) + \frac{\alpha^2}{2R_i} \bigg]\dd t  +\alpha\,\dd W_{R_i} , 
	\label{eq:2r}\\
	\dd \phi & = -\beta \bigg( \frac{R_k}{R_i} + \frac{R_i}{R_k} \bigg)\sin{\phi} \,\dd t + \alpha \bigg( \frac{1}{R_i}\dd W_{\theta_i}  - \frac{1}{R_k}\dd W_{\theta_k} \bigg) .
	\label{eq:2phi}
\end{align}
where $\phi = \theta_i-\theta_k$ is the phase difference between the two nodes and $W_{R_i}$, $W_{\theta_i}$ and $W_{\theta_k}$ are independent Wiener processes for $i,k\in\{1,2\}$.
The subsystem~\eqref{eq:2r} can be written as a noise driven potential system
\begin{equation}
\frac{\dd }{\dd t} R_1  = - \frac{\partial V}{\partial R_1},~~~~
\frac{\dd }{\dd t} R_2  = - \frac{\partial V}{\partial R_2},~~~~
\label{eq:2rpot}
\end{equation}
for the potential
\begin{equation}
	V=\frac{1}{2} \bigg[ \frac{R_1^6+R_2^6}{3} - (R_1^4 + R_2^4) + (\nu+\beta)(R_1^2 + R_2^2) - \alpha^2\ln(R_1R_2)\bigg] - \beta R_1R_2 \cos \phi.
	\label{eq:2biV}
\end{equation}
The phase difference $\phi$ governed by (\ref{eq:2phi}) has, for $\beta>0$ and the low noise limit $\alpha\rightarrow 0$, a stable fixed point at $\phi=0$ and an unstable fixed point at $\phi=\pi$ if the $R_{i}$ are bounded away from zero. 
All local minima of the potential \eqref{eq:2biV} will have $\phi=0$,
as will all saddles that act as gates between basins of attraction. Hence we restrict to $\phi=0$ from hereon and analogously to \cite{act17domino} we perform a bifurcation analysis of the noise-free version of (\ref{eq:2r}), namely
\begin{align}
\frac{\dd}{\dd t} R_i &= -(\nu+\beta)R_i + 2R_i^3 - R_i^5 + \beta R_k + \frac{\alpha^2}{2R_i} , 
\label{eq:2rnonoise}
\end{align}
$\nu=0.2$ and $\alpha=0.05$. \Fref{fig:2bibif} shows the bifurcation diagram of $\beta$ against $R_1$, analogously to \cite{act17domino}. Panels (b)--(d) show the $(R_1,R_2)$-plane with the equilibria of \eqref{eq:2rnonoise} and potential contours of \eqref{eq:2biV}  for $\phi=0$. Each panel (b)--(d) also shows a typical realisation of~\eqref{eq:neteq} computed in {\sc{Matlab}} using the Heun method with initial point $z_1=z_2=0$.

The bifurcation diagram of (\ref{eq:2rnonoise}) depicted in \Fref{fig:2bibif}(a) is computed in {\sc Auto}~\cite{AutoOrig} and shows two symmetric, simultaneous saddle-node (SN) bifurcations at $\beta_{\rm{SN}}=0.0154297$, the second of which is difficult to discern as three equilibria (the saddle and sink involved in the SN and the sink of the deepest well) have very close values in $R_1$.
There is a pitchfork bifurcation at $\beta_{\rm{PF}}=0.164917$.
These bifurcations split the diagram into three regimes:
\begin{itemize}
\item $0<\beta<\beta_{\rm{SN}}$ has nine equilibria, one source, four sinks
and four saddles.
\item $\beta_{\rm{SN}}<\beta<\beta_{\rm{PF}}$ has five equilibria; one source, two
saddles and two sinks.
\item $\beta>\beta_{\rm{PF}}$ has three equilibria; two sinks and one
saddle.
\end{itemize}
As in \cite{act17domino} the first regime corresponds to {\em weak coupling}, the second to {\em intermediate coupling} and the third to {\em strong coupling} (see also \cite{BFG2007a,BFG2007b}). The remainder of this paper examines the influence of these regimes on the escape times.

\begin{figure}
\centering
\includegraphics[width=0.95\textwidth]{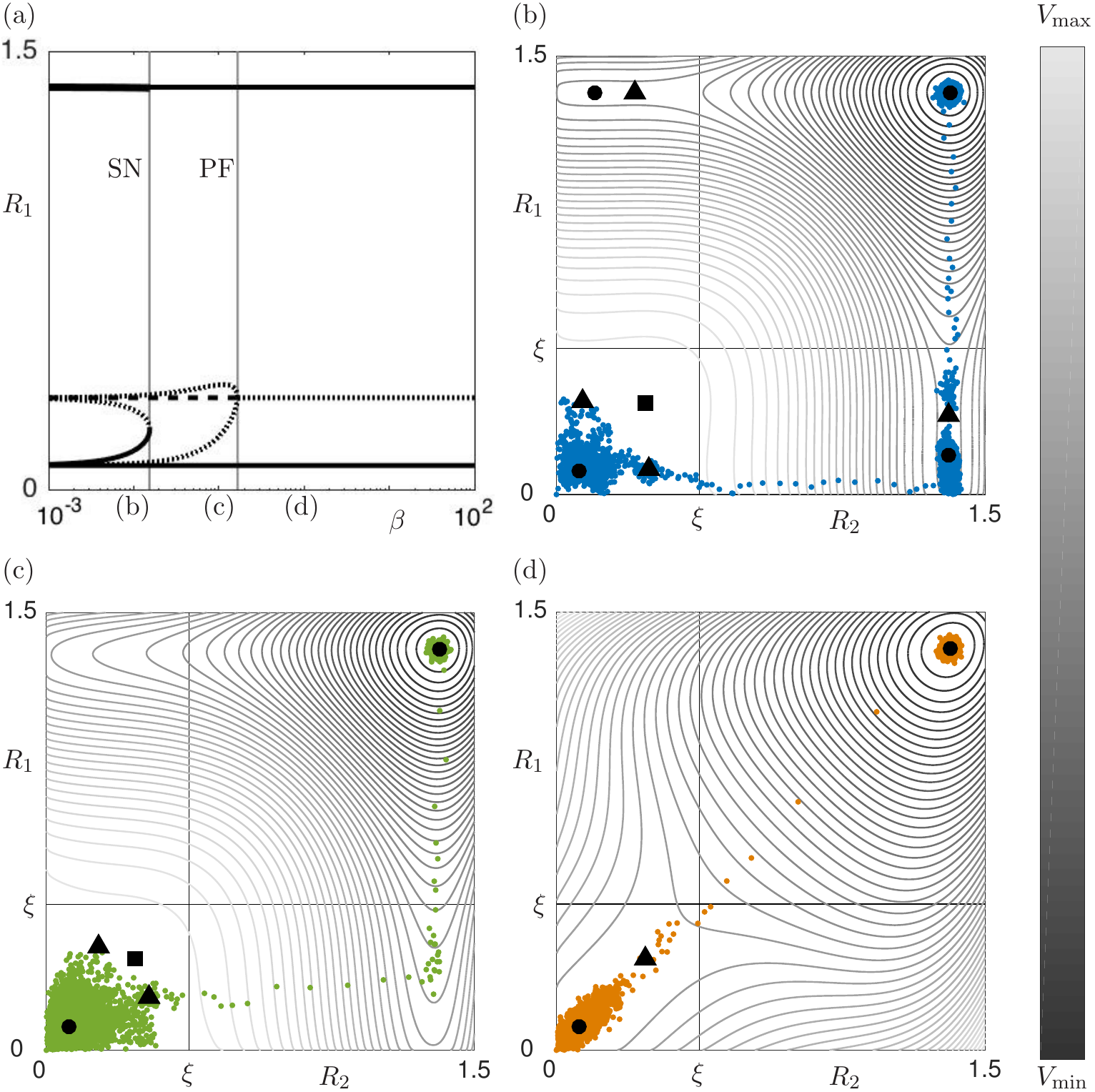}
\caption{Bifurcation diagram and corresponding phase portraits for $\nu=0.2$ and  $\alpha=0.05$.  Panel (a) shows the bifurcation diagram of $R_1$ against $\beta$ for \eqref{eq:2rnonoise}.  There are three regimes separated by two simultaneous saddle-node (SN) bifurcations and a pitchfork bifurcation (PF). Panels (b)--(d) show typical realisations plotted with contour lines of  potential $V$ for $\phi=0$ at $\beta$ values representative of each coupling regime. Specifically, for the weak coupling regime $\beta=0.01$ (b), for the intermediate coupling regime $\beta=0.1$ (c) and for the strong coupling regime $\beta=1$ (d). Equilibria of \eqref{eq:2rnonoise} are shown as $\bullet$  for sinks, $\blacksquare$  for sources and $\blacktriangle$  for saddles. The straight lines $R_i=\xi=0.5$ show the thresholds used to quantify escapes.
}
\label{fig:2bibif}
\end{figure}

\subsection{Estimating escape times for two coupled nodes}
\label{sec:2est}

We numerically compute $T_2^{1|0}$, $T_2^{2|0}$ and $T_2^{2|1}$ in {\sc{Matlab}} by fixing example parameters $\nu=0.2$ and $\alpha=0.05$ and computing an ensemble of $2000$ realisations of \eqref{eq:neteq} for $i=2$ using the Huen method with step size $h=10^{-3}$. Note, we set $\omega=0$ in \eqref{eq:neteq} as the radial dynamics do not depend on $\omega$. We choose threshold $\xi=0.5$ to determine the escape times, as shown in \Fref{fig:2bibif}(b)--(d).  The first and second escape times $\tau_2^{1|0}$ and $\tau_2^{2|0}$ are averaged over the ensemble to give numerical approximations of $T_2^{1|0}$ and $T_2^{2|0}$, while $T_2^{2|1} =T_2^{2|0} - T_2^{1|0} $. 

Using numerical integration of the one-node case \eqref{eq:meanT}, we first estimate the limits for $\beta\rightarrow 0$ and $\beta\rightarrow\infty$ for bidirectional coupling. In the infinite coupling limit, the two systems are strongly synchronised and act as a single node with the same $\nu$ but attenuated noise, $\alpha/\sqrt{2}$. For $\nu=0.2$, $\alpha=0.05$ and $\xi=0.5$ in the limit $\beta\rightarrow \infty$ we expect
\begin{align}
T^{1|0}_2&\rightarrow {\cal{T}}_{(1)}=T\left(\nu, \alpha/\sqrt{2}\right)\approx6312.21, \label{eq:lim1} \\
T^{2|1}_2&\rightarrow 0. \nonumber 
\end{align}
The uncoupled limit has independence of escapes so the approximate mean escape time is half the mean escape time of one node. For $\nu=0.2$, $\alpha=0.05$ and $\xi=0.5$ in the limit $\beta\rightarrow 0$ (or in the uncoupled case for all $\beta$) we expect
\begin{align}
T^{1|0}_2&\rightarrow {\cal{T}}_{(2)}=\frac{T(\nu,\alpha)}{2}\approx 96.51, \label{eq:lim2}\\
T^{2|1}_2&\rightarrow {\cal{T}}_{(3)}=T(\nu,\alpha) \approx 193.01.  \label{eq:lim3}
\end{align}
We also note that for the unidirectional case, the first escape $T_2^{1|0}$ will be either from the driving node with mean $T(\nu,\alpha)$ or from the driven node with mean $T(\nu,\alpha/\sqrt{2})$. The sum of the rates of escape corresponds in the limit $\beta\rightarrow \infty$ to
\begin{equation}
T_2^{1|0}\rightarrow {\cal{T}}_{(4)}= \frac{T(\nu,\alpha)T(\nu, \frac{\alpha}{\sqrt{2}})}{T(\nu,\alpha)+T(\nu, \frac{\alpha}{\sqrt{2}})}=188.01  \label{eq:lim4}
\end{equation}

For the bidirectionally coupled case, we also find that many features of the scalings of first passage times $T^{1|0}$ and $T^{2|1}$ can be understood from the coupling regimes of the deterministic potential system \eqref{eq:2rpot}. We estimate these scalings in the three coupling regimes using generalized Eyring-Kramers Laws~\cite{berg08EK,berg13kramers} for saddles that are multidimensional and/or passing through bifurcation. 

For the given value of $\nu$ and $\alpha$ we relate the Eyring-Kramers times $T_K$ to the numerical experiments $T$ using a common linear transformation of the form
$$
T \simeq A T_K +B
$$
where the constants $A,B$ are determined from the one-node case. Specifically, they are the unique solution such that the one-node estimates for $\mathcal{T}_{(1)}$ and $\mathcal{T}_{(2)}$ are correct, namely
$$
T(\nu,\alpha)=A T_{K}(\nu,\alpha)+B,~~T(\nu,\alpha/\sqrt{2})=A T_{K}(\nu,\alpha/\sqrt{2})+B.
$$
so that $A$ and $B$ do not depend on $\beta$. For $\nu=0.2$ and $\alpha=0.05$ we find $A=4.38$ and $B=-295$.

\paragraph{First escape time statistics}

The mean first escape time $T^{1|0}_2$ is associated with escape over one of two possible saddles for weak and intermediate coupling, $\beta<\beta_{\text PF}$. These saddles merge into a single synchronised saddle for strong coupling, $\beta>\beta_{\rm{PF}}$, where $T^{1|0}_2$ is associated with escape over the only remaining saddle. 
A multidimensional Eyring-Kramers Law~\cite{berg13kramers} gives an asymptotic approximation $\widehat{T}_2^{1|0}$ for $T_2^{1|0}$.
Denote by $x$ the potential minimum where neither node has escaped and denote by $y$ one of the two saddles that undergo the pitchfork bifurcation, or the only saddle for $\beta>\beta_{PF}$.
Then we compute
\begin{equation}
\widehat{T}_2^{1|0}(\beta) = \frac{2\pi}{|\lambda_1(y)|}\sqrt{\frac{|{\rm{det}}(\nabla^2V(y))|}{{\rm{det}}(\nabla^2V(x))}}{\rm{e}}^{[V(y)-V(x)]/\varepsilon}
\label{eq:EKfirst}
\end{equation}
where $\varepsilon=\alpha^2/2$ and  $\lambda_1(y)$ is the single negative eigenvalue of the Hessian $\nabla^2V(y)$. Here we use the two node potential $V$ given by \eqref{eq:2biV} with $\alpha=0.05$, $\nu=0.2$ and $\phi=0$. This breaks down for $\beta \to \beta_{PF}$ where $\lambda_1(y)\to 0$. 
Berglund and Gentz~\cite{berg08EK} derive a multidimensional Eyring-Kramers Law for escapes from a potential well over a saddle that undergoes pitchfork bifurcation. This corresponds to the case as $\beta$ passes through $\beta_{\text PF}$ and gives an asymptotic expression $\widetilde{T}_2^{1|0}$ for $T_2^{1|0}$ on either side of the pitchfork bifurcation. Denote by $z_{\pm}$ the two saddles for $\beta<\beta_{PF}$ that merge at the pitchfork bifurcation, and by $z$ the saddle for $\beta>\beta_{\text PF}$. Denote by $\mu_1<0<\mu_2$ the eigenvalues of $\nabla^2V(z_{\pm})$ and by $\lambda_1<0<\lambda_2$ the eigenvalues of $\nabla^2V(z)$. Finally we let $x$  be as before.
Then by \cite[Corollary 3.8]{berg08EK} and noting $\varepsilon=\alpha^2/2$:
\begin{equation}
\widetilde{T}_{2}^{1|0}(\beta) = 
\left\{
\begin{array}{cl}
\displaystyle{2\pi \sqrt{\frac{\mu_2 +(2\varepsilon C_4)^{1/2}}{|\mu_1|{\rm{det}}(\nabla^2V(x))}}\frac{{\rm{e}}^{[V(z_{\pm})-V(x)]/\varepsilon}}{\Psi_-(\mu_2/(2\varepsilon C_4)^{1/2})}[1+E_-(\varepsilon,\mu_2)]}
& \mbox{ for }\beta<\beta_{\text PF}\\
\displaystyle{ 2\pi \sqrt{\frac{\lambda_2 +(2\varepsilon C_4)^{1/2}}{|\lambda_1|{\rm{det}}(\nabla^2V(x))}}\frac{{\rm{e}}^{[V(z)-V(x)]/\varepsilon}}{\Psi_+(\lambda_2/(2\varepsilon C_4)^{1/2})}[1+E_+(\varepsilon,\lambda_2)]}& \mbox{ for }\beta\geq\beta_{\text PF}
\end{array}\right.
\label{eq:EKTfirst}
\end{equation}
The coefficient $C_4>0$ represents the coefficient of the quartic expansion near the bifurcation point and the $\Psi_{\pm}$ are given in \cite{berg08EK} as
\begin{align*}
\Psi_+(\gamma) &= \sqrt{\frac{\gamma(1+\gamma)}{8\pi}}{\rm{e}}^{\frac{\alpha^2}{16}}J_{1/4}\left( \frac{\alpha^2}{16} \right),\\
\Psi_-(\gamma) &=  \sqrt{\frac{\pi \gamma(1+\gamma)}{32}}{\rm{e}}^{-\frac{\alpha^2}{64}}\left[ I_{-1/4}\left( \frac{\alpha^2}{64} \right) +I_{1/4}\left( \frac{\alpha^2}{64} \right) \right].
\end{align*}
The Bessel functions $I_{\pm 1/4}$ and $J_{1/4}$ are of the first and second kinds, respectively. The error functions $E_{\pm}$ are bounded and tend to zero in $\varepsilon\rightarrow 0$ for some neighbourhood of $\lambda_1=0$: we set $E_{\pm}=0$ to define $\widetilde{T}^{1|0}_2(\beta)$.
The quantities in (\ref{eq:EKTfirst}) are available in terms of properties of the potential $V$ and so numerically computable from the parameters. 

\paragraph{Second escape time statistics}

The mean second escape $T^{2|1}_2$ has three, rather than two, identifiable regimes. For $\beta<\beta_{SN}$ it is associated with noise-induced escape from the attracting state where only one of the nodes has escaped. Here we again use the multidimensional Eyring-Kramers Law~\cite{berg13kramers} to find an asymptotic approximation $\widehat{T}_2^{2|1}$ for $T_2^{2|1}$ for $0\leq \beta<\beta_{SN}$. Specifically, let $x$ and $z$ be the sink and saddle respectively that undergo the saddle-node bifurcation at $\beta_{SN}$. Using the two node potential $V$ given by \eqref{eq:2biV} with $\alpha=0.05$, $\nu=0.2$ and $\phi=0$, we compute
\begin{equation}
\widehat{T}_2^{2|1}(\beta) = \frac{2\pi}{|\lambda_1(z)|}\sqrt{\frac{|{\rm{det}}(\nabla^2V(z))|}{{\rm{det}}(\nabla^2V(x))}}{\rm{e}}^{[V(z)-V(x)]/\varepsilon}
\label{eq:EKTsec}
\end{equation}
where $\varepsilon=\alpha^2/2$ and  $\lambda_1(z)$ is the single negative eigenvalue of the Hessian $\nabla^2V(z)$. This is valid for $\beta\ll\beta_{SN}$, but it breaks down for $\beta \to \beta_{SN}$ where $\lambda_1(z)\to 0$. The approximations \eqref{eq:EKTfirst} and \eqref{eq:EKTsec} are asymptotic for small $\alpha$. 

For $\beta_{SN}<\beta<\beta_{PF}$ the second escape is associated with following the unstable manifold from one of the two saddles that exist in the region.
In this case, the second escape does not pass through a second potential well of \eqref{eq:2biV} and so the escape times $\tau_2^{2|1}$ are no longer exponentially distributed beyond this point. For example, if we assume $R_1$ escapes before $R_2$ over the saddle $x$ then we can numerically estimate $T^{2|1}_2$ by parametrizing the unstable manifold $W^u(x)$ by $(r_1(t),r_2(t))$ between $R_1=\xi$ and $R_2=\xi$,
\begin{equation}
\widetilde{T}_2^{2|1} = \inf\{t~:~r_1(s)=\xi=r_2(t+s)~\mbox{ for some }s\}.
\label{eq:T21approx1}
\end{equation}
Specifically, we compute orbit segments that lie on the unstable manifold $W^u(x)$ as solutions of a multi-segment boundary value problem set-up in the software package {\sc{Auto}} \cite{AutoOrig,auto}; for general theory of computation of manifolds with AUTO see \cite{doedel07lec,kraus07comp}. We rescale the deterministic part of equations \eqref{eq:2r} so that the integration time becomes a parameter of the system. We fix the parameters $\nu=0.2$, $\alpha=0.05$, $\phi=0$ and $\beta=0.0155$, very close to but just past the saddle node bifurcation. We  compute one orbit segment that has one end in the linear unstable direction of $x$ and the other at the threshold $R_1=\xi$. We then compute a second orbit segment that has one end equal to the end of the first segment and the other end of the orbit segment lies on the threshold $R_2=\xi$. We continue the system with $\beta$ as the main increasing continuation parameter up to $\beta_{PF}$ whilst monitoring the integration time of the second orbit segment. Here we use variable step size $10^{-6}<h<1$ and $\text{\sc{ntst}}=300$ mesh points. 

For $\beta_{PF}<\beta$ the second escape is associated with fluctuations away from synchrony near the synchronous unstable manifold of the single saddle. In this case we can estimate the scaling with $\beta$ by considering these fluctuations.
For some constant $\delta$ and threshold $\xi$, we approximate the dynamics through the region
$$
R_1=R+\delta,R_2=R-\delta.
$$
For trajectories passing by $R=\xi$, $\delta=0$ we have for $\nu=0.2$ and $\alpha=0.05$
\begin{align*}
\dot{R} &\approx 0.12125\\
\dot{\delta}&\approx - L \delta + \alpha dW_t
\end{align*}
Where the value of $\dot{R}$ is given at $R=\xi=0.5$ and $\delta=0$, and $L=(2\beta-0.329834)$ is the transverse eigenvalue at the saddle; note, $L=0$ at the PF bifurcation.
If we assume that the process for $\delta$ is in equilibrium then $\delta\sim N(0,\alpha^2/(2L))$.
This means the first escape will happen approximately at time $T$ before $R=\xi$ such that $0.12125 T \approx \alpha/(\sqrt{2L})$. The second escape will happen at a time roughly the same time after $R=\xi$, meaning we have an estimate
\begin{equation}
\widetilde{T}_2^{2|1} \sim 2T = (\alpha/\Delta) \sqrt{2/L}= (0.05/0.12125)\sqrt{2/(2\beta-0.329834)}
\label{eq:T21approx2}
\end{equation}

\begin{figure}[t] 
\includegraphics[width=\textwidth]{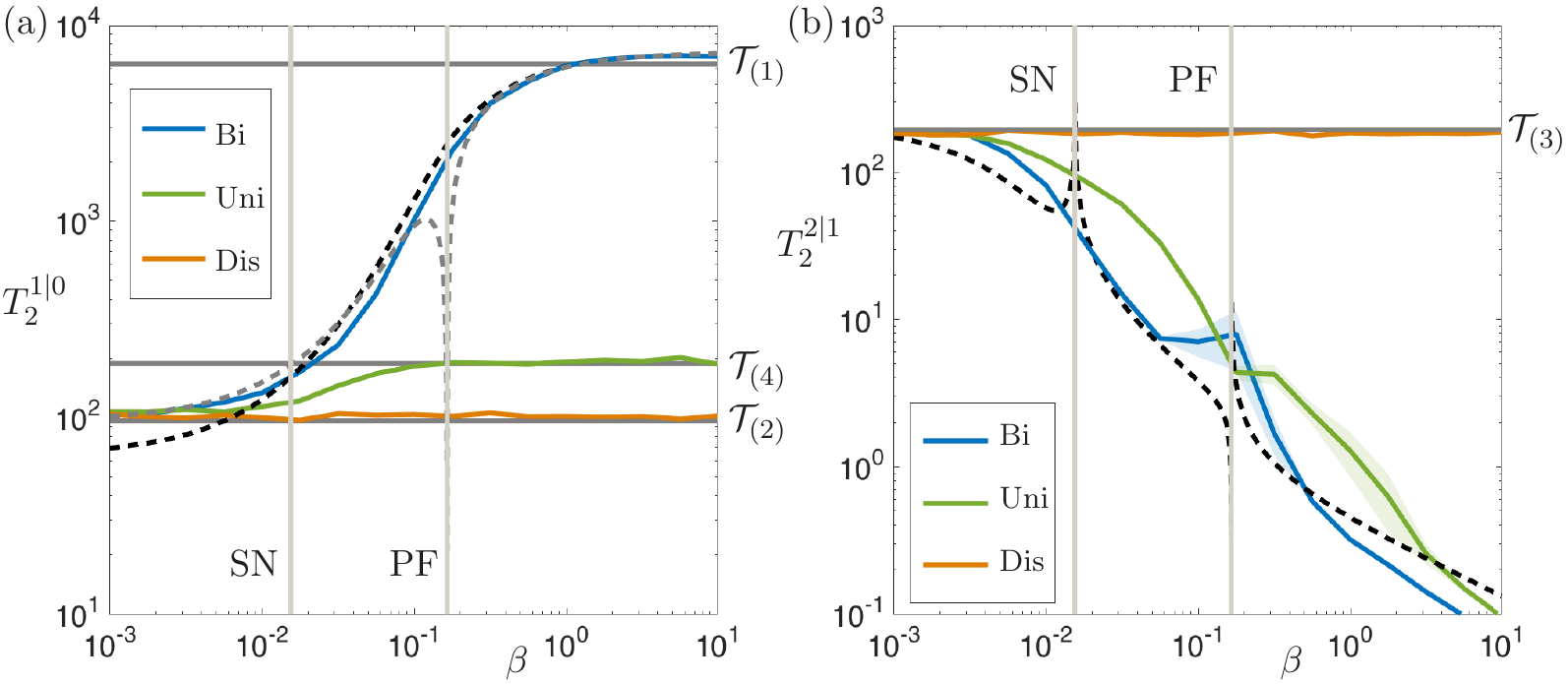}
\caption{Numerically computed mean first passage times $T_2^{1|0}$ (a) and $T_2^{2|1}$ (b) against $\beta$ for two nodes with bidirectional coupling (Bi), unidirectional coupling (Uni) and disconnected (Dis) for $\nu=0.2$, $\alpha=0.05$ and threshold $\xi=0.5$. The grey lines indicate the saddle-node  (SN) and the pitchfork  (PF) bifurcations in the case of bidirectional coupling. Panel (a) shows mean first passage times $T_2^{1|0}$  (cf.~ \cite[Figure~6(b)]{benj12pheno}) with $A\widehat{T}_2^{1|0}+B$ (grey dashed) where $\widehat{T}_2^{1|0}$ is given by \eqref{eq:EKfirst} and $A\widetilde{T}_2^{1|0}+B$ (black dashed) where $\widetilde{T}_2^{1|0}$ is given by \eqref{eq:EKTfirst}. Here $A=4.38$ and $B=-295$. The predicted asymptotic escape times ${\cal{T}}_{(1)},{\cal{T}}_{(2)}$ and ${\cal{T}}_{(4)}$ for each network in the limit $\beta\rightarrow \infty$ are shown. Note that in the limit $\beta\rightarrow 0$, all times limit to ${\cal{T}}_{(2)}$. Panel (b) also shows $A\widehat{T}_2^{2|1}+B$ (black dashed)  where $\widehat{T}_2^{2|1}$ is given by $\eqref{eq:EKTsec}$ and $\widetilde{T}_2^{2|1}$  (black dashed) given by \eqref{eq:T21approx1}--\eqref{eq:T21approx2}. The predicted asymptotic escape time ${\cal{T}}^{(3)}$ is also plotted and all times limit to ${\cal{T}}_{(3)}$ for $\beta \to 0$.  The error bars for $T_2^{2|1}$ for the bidirectional and unidirectional coupling become much larger around  $\beta_{PF}$ due to the small escape times.
}
\label{fig:2escT}
\end{figure}

Figure~\ref{fig:2escT} shows numerical simulations of $T_2^{1|0}$ and $T_2^{2|1}$ plotted with error bars against $\beta$ for two nodes with bidirectional coupling, unidirectional coupling and disconnected. 
We mark  the estimated limits ${\cal{T}}_{(i)}$ for $i=1...4$ given by \eqref{eq:lim1}--\eqref{eq:lim4} computed for $\nu=0.2$, $\alpha=0.05$ and $\xi=0.5$ by numerical integration of \eqref{eq:meanT}. We mark the location of the saddle-node and pitchfork bifurcations in each panel. Panel (a) shows $A\widehat{T}_2^{1|0}+B$ and $A\widetilde{T}_2^{1|0}+B$ where $\widehat{T}_2^{1|0}$ and $\widetilde{T}_2^{1|0}$ are the classic and modified multidimensional Eyring-Kramers times given by \eqref{eq:EKfirst} and \eqref{eq:EKTfirst} respectively. Panel(b) shows $A\widehat{T}_2^{2|1}+B$ where  $\widehat{T}_2^{2|1}$ is the multidimensional Eyring-Kramers times given by \eqref{eq:EKTsec}, and the additional estimates  $\widetilde{T}_2^{2|1}$ given by \eqref{eq:T21approx1} and \eqref{eq:T21approx2} are also shown. 

In both panels the estimated limits agree well with those found numerically for the full system. In panel (a) the approximation $A\widehat{T}_2^{1|0}+B$ has a clear asymptote at the pitchfork bifurcation, whereas, the two curves $A\widetilde{T}_2^{1|0}+B$ meet at $\beta_{PF}$ and note that the general shape is consistent with the numerically computed $T_2^{1|0}$.  Panel (b) shows $A\widehat{T}_2^{2|1}+B$ is close to $T_2^{2|1}$ for $\beta=10^{-3}$ but the two curves diverge rapidly as $\beta \to \beta_{SN}$. Approximations $\widetilde{T}_2^{2|1}$ also diverge from $T_2^{2|1}$ at the bifurcations, but follow the general shape of $T_2^{2|1}$ for $\beta_{SN} << \beta<<\beta_{PF}$ and $\beta>\beta_{PF}$. The accuracy of our numerical simulations also breaks down for times around $10^{1}$ as can be seen from the large error bars around the bidirectional and unidirectional curves. This is due to small escape times and the fixed step size used in our computations.

\section{A master equation approach to sequential escape}
\label{sec:master}

In this section we use a Master equation approach to model sequential escape on a network of $N$ nodes. We assume that each node can undergo a transition from quiescent state $0$ to active state $1$ as a memoryless Markov jump process, and as above we assume there are no transitions from $1$ back to $0$ meaning the associated Markov chain is transient to an absorbing state. Using this, at least in the weak coupling limit, we find good agreement not only to the mean sequential escape times but also to their distributions.

\subsection{Sequential escapes for weak coupling}

We characterise the state of a system such as \eqref{eq:neteq} at time $t$ by the vector
$$
X(t)=\{X_i\}\in \{0,1\}^N.
$$
where node $X_j$ changes from $0$ to $1$ according to a memoryless process with rate $r^{X}_j$: this rate may depend on current state $X$. For a fixed threshold $\xi>0$ one may think of $X_i(t)$ as the discrete random variable that changes from zero to one at the first escape time $\tau^{(i)}$ of node $i$, i.e.
$$
X_i(t)=\left\{ \begin{array}{cl}
0 & \mbox{ if }0\leq t<\tau^{(i)}\\
1 & \mbox{ otherwise}
\end{array}\right.
$$

The set of such states forms the vertices of an $N$-dimensional hypercube and there is a directed edge if a transition is possible from one state to the other. The irreversible nature of the process means there is a directed cycle-free sequence of transitions on $\{0,1\}^N$ that leads from $X(0)=\{0\}^N$ to $X = \{1\}^N$. 
Figure~\ref{fig:hyper23} shows possible states of (a) two node and (b) three node networks as vertices of two- and three-dimensional hypercubes, respectively.
The directed edges are labelled with their corresponding transition rates; for example, in panel (a) the rate $r_2^{\{0,0\}}$ is the transition rate of node 2 given that we are in state $X=[0,0]$.

\begin{figure}
\centering
\includegraphics[width=13.5cm]{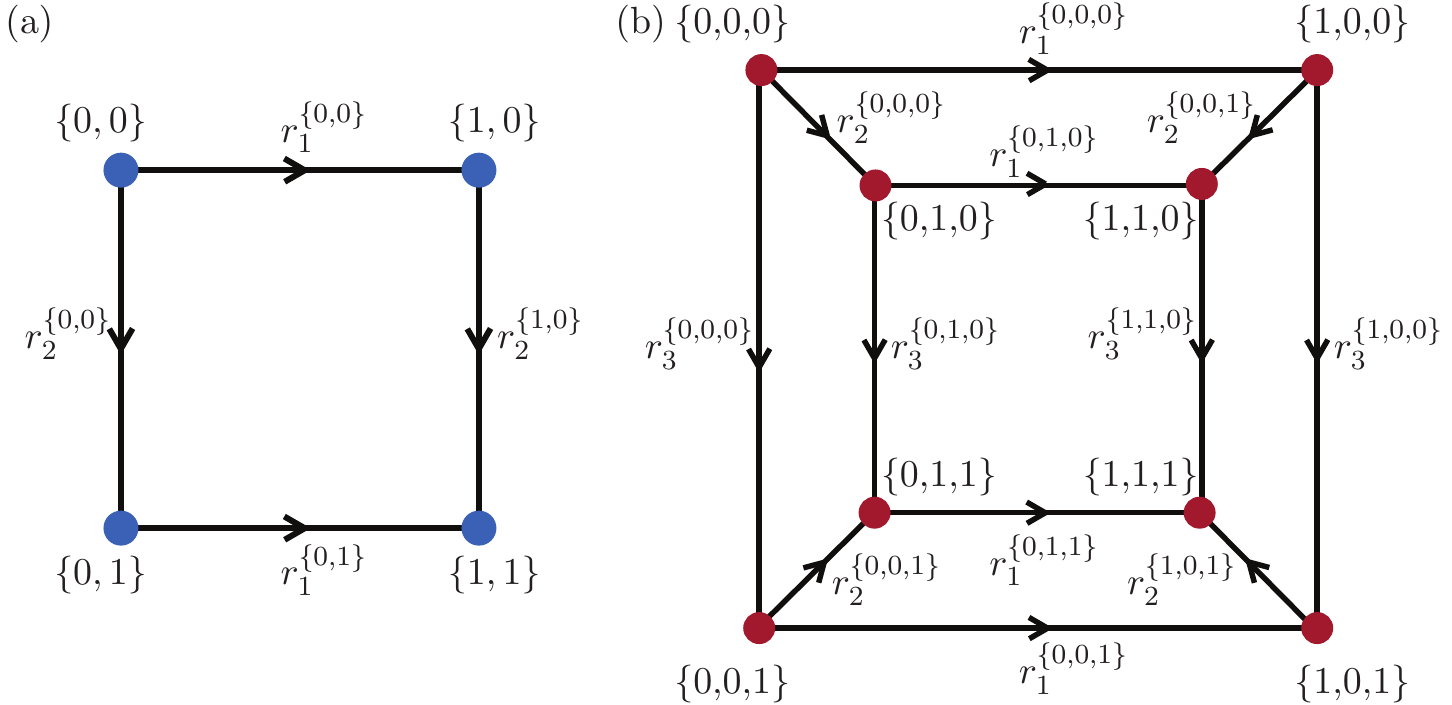}
\caption{Hypercube of states for a network of (a) two nodes and (b) three nodes. Each vertex of the hypercube is a state $X$ of the network and the directed edges indicate permissible transitions between states at state dependent transition rates $r_j^X$. }
\label{fig:hyper23}
\end{figure}

Let $p_{X}(t)$ represent the probability that the system is in state $X$ at time $t>0$: we study $p_{X}(t)$ using a master equation approach. Define the {\em origin} operator $O_j:\{0,1\}^N\rightarrow \{0,1\}^N$ by
\begin{equation*}
[O_j(X)]_k:=\left\{ \begin{array}{rl}
X_k & \mbox{ if }k\neq j\\
0 & \mbox{ if }k=j.
\end{array}
\right.
\end{equation*}
The probability $p_{X}$ satisfies the master equation:
\begin{equation}
\frac{d}{dt} p_{X} = \sum_{j~:~X_j=1} r^{O_j(X)} p_{O_j(X)}-  \sum_{j~:~X_j=0} r^{X}_j p_{X}
\end{equation}
for $X\in \{0,1\}^N$. The first sum represents the rate at which states arrive at state $X$ and the second sum represents the rate at which they leave state $X$. This gives a system of linear ordinary differential equations (ODE) that can be represented as
\begin{equation}
\frac{\dd}{\dd t} p = M p
\label{eq:meq}
\end{equation}
for the $2^N$ dimensional vector $p(t)$, where $M$ is a $2^N$ square matrix.  For example, the two node network shown in \Fref{fig:hyper23}(a) is governed by
\begin{equation*}
M =
\begin{bmatrix}
-r_1^{[0,0]} - r_2^{[0,0]} & 0 & 0 & 0 \\
r_1^{[0,0]} & -r_2^{[1,0]} & 0 & 0 \\
r_2^{[0,0]} & 0 & -r_1^{[0,1]} & 0\\
0 & r_2^{[1,0]} & r_1^{[0,1]} & 0
\end{bmatrix}
  \ \ {\rm{and }}   \ \ 
p = 
\begin{bmatrix}
p_{[0,0]}\\ p_{[1,0]}\\ p_{[0,1]}\\ p_{[1,1]}
\end{bmatrix}
\end{equation*}
with solution $p(t)= p(0) \exp Mt$. The eigenvalues of $M$ are $\lambda_1=-r_1^{[0,0]} - r_2^{[0,0]}$, $\lambda_2=- r_2^{[1,0]}$, $\lambda_3=- r_1^{[0,1]}$ and $\lambda_4=0$ with corresponding eigenvectors $v_i$. In particular, the unique zero eigenvalue $\lambda_4$ has eigenvector $v_4=[0, 0, 0, 1]^T$ showing that the state $X=[1,1]$ is an absorbing state for the system.

\subsection{The all-to-all coupled case}
\label{sec:evol}

We define the probability that $k$ out of $N$ nodes have escaped by time $t$ to be
\begin{equation}
p_{N,k} (t):= {\mathbb{P}}\{ |\{i~:~X_i(t)=1\}|=k\}.
\label{eq:pNkdef}
\end{equation}
Explicit formulae for $p_{N,k}$ can be found for the all-to-all coupled case, where the rate $r^X_j$ depends only on the number of escaped nodes, i.e. where
$$
r_{j}^X = r_{|\{i~:~X_i=1\}|}
$$
and the rate $r_i$ corresponds to the rate at which the remaining nodes escape, given that $i\in\{0,\ldots,N\}$ of them have already escaped (we use the convention that $r_N=0$). 

For example, for a two node network note that $p_{2,0}=p_{[0,0]}$, $p_{2,1} = p_{[0,1]} +  p_{[1,0]}$ and $p_{2,2}=p_{[1,1]}$.
In the uncoupled case with rate $r$ at each node, substituting in $N=2$ and $n=0,1,2$ we obtain
\begin{align*}
p_{2,0} & =  {\rm{e}}^{-2rt},\\ 
p_{2,1}&  =  {\rm{e}}^{-rt}(1- {\rm{e}}^{-rt}), \\ 
p_{2,2}& = -2  {\rm{e}}^{-rt}+ {\rm{e}}^{-2rt}+1.
\end{align*}

More generally, for the all-to-all coupled case the $2^N$ equations of \eqref{eq:meq} for the $p_{X}$ can be reduced to $N+1$ equations for the $p_{N,k}$, where
$\sum_{k=0}^N p_{N,k}=1$. The resulting linear system has $N+1$ eigenvalues
\begin{equation}
 \lambda_k = -(N-k)r_k
 \label{eq:geneigris}
\end{equation}
for $k=0, \dots, N$. As the non-zero off diagonal entries are $-\lambda_k$, the $k^{\rm{th}}$ equation is
\begin{equation}
\frac{\dd }{\dd t} p_{N,k} =- \lambda_{k-1}p_{N,k-1} + \lambda_k p_{N,k}.
\label{eq:genris}
\end{equation}

\begin{prop}
Assume that $\lambda_k<0$ for $k=1,\ldots,N-1$ are distinct and all nodes are quiescent at time $t=0$, i.e. $p_{N,0}(0)=1$. Then the solution of \eqref{eq:genris} is given by
\begin{equation}
p_{N,k} (t) = \left[\prod_{i=0}^{k-1} \lambda_i\right] \left[ \sum_{j=0}^k \frac{{\rm{e}}^{\lambda_j t}}{\prod_{n \neq j} (\lambda_n -\lambda_j)} \right].
\label{eq:PNk}
\end{equation}
\label{prop:master}
\end{prop}

\begin{proof}  
We show this by induction for any $N>0$ and all integers $0\leq k\leq N$. It is clear that $\frac{\dd}{\dd t} p_{N,0} = \lambda_0 p_{N,0}$ has the solution $p_{N,0}={\rm{e}}^{\lambda_0 t}$, for any $N$.
It follows that $\frac{\dd}{\dd t} {p}_{N,1} = -\lambda_0 p_{N,0} + \lambda_1 p_{N,1}$ has the solution
$$
p_{N,1} = \frac{\lambda_0}{\lambda_1 - \lambda_0}{\rm{e}}^{\lambda_0 t} + \frac{ \lambda_0}{\lambda_0 - \lambda_1}{\rm{e}}^{\lambda_1 t},
$$
and $\frac{\dd}{\dd t} {p}_{N,2} = - \lambda_1 p_{N,1} + \lambda_2 p_{N,2}$ has the solution
$$
p_{N,2} = \frac{\lambda_0 \lambda_1}{(\lambda_0 - \lambda_1)(\lambda_0-\lambda_2)}{\rm{e}}^{\lambda_0 t} + \frac{\lambda_0 \lambda_1}{(\lambda_1 - \lambda_0)(\lambda_1 - \lambda_2)}{\rm{e}}^{\lambda_1 t}  + \frac{\lambda_0 \lambda_1}{(\lambda_2 - \lambda_0)(\lambda_2 - \lambda_1)}{\rm{e}}^{\lambda_2 t}.
$$
Now assume that the result holds for some $k<N$. Using \eqref{eq:genris} we write $\dot{p}_{N,k+1} (t) = -\lambda_k p_{N,k} + \lambda_{k+1} p_{N,k+1}.$
Using integration factor $\text{e}^{-\lambda_{k+1}t} $ and~\eqref{eq:PNk} gives
\begin{align*}
\frac{\text{d}}{\text{d}t}\left( p_{N,k+1} \text{e}^{-\lambda_{k+1}t}  \right) &= -\lambda_{k}p_{N,k}\text{e}^{-\lambda_{k+1}t},\\
&=-\prod_{i=0}^{k} \lambda_i \left[ \sum_{j=0}^k \frac{{\rm{e}}^{(\lambda_j -\lambda_{k+1})t}}{\prod_{n \neq j} (\lambda_n -\lambda_j)} \right].
\end{align*}
Integrating both sides with respect to $t$ we obtain
\begin{equation*}
p_{N,k+1}e^{-\lambda_{k+1}t}=-\prod_{i=0}^k \lambda_i \left[ \sum^{k}_{j=0}  \frac{{\rm{e}}^{(\lambda_j -\lambda_{k+1})t}}{(\lambda_j - \lambda_{k+1})\prod_{n \neq j} (\lambda_n -\lambda_j)} + C \right],
\end{equation*}
and so
\begin{equation*}
p_{N,k+1}=\prod_{i=0}^k \lambda_i \left[ \sum^{k}_{j=0} {\rm{e}}^{\lambda_{j}t}\prod_{n=0,n \neq j}^{k+1}  \frac{1}{(\lambda_n -\lambda_j)} + C{\rm{e}}^{\lambda_{k+1}t} \right],
\end{equation*}
Using the initial condition $X(0)=\{0\}^N$, equally $p_{N,k+1}(0) = 0 $, we find
\begin{align*}
C &= -\sum_{j=0}^k  \prod_{n=0,n \neq j}^{k+1}  \frac{1}{(\lambda_n -\lambda_j)}=\prod_{n=0}^{k}  \frac{1}{(\lambda_n -\lambda_{k+1})}
\end{align*}
which follows from the identity
\begin{equation}
\sum_{j=0}^{k+1} \prod_{n=0,n\neq j}^{k+1} \frac{1}{(\lambda_n-\lambda_j)}=0.
\label{eq:identity}
\end{equation}
Equation (\ref{eq:identity}) can be shown by considering the order $k$ Lagrange interpolating polynomial 
$$
P(x)=\sum_{j=0}^{k+1} \prod_{n=0,n\neq j}^{k+1} \frac{(x-\lambda_j)}{(\lambda_n-\lambda_j)}
$$ 
which is equal to $1$ at the $k+1$ points $x=\lambda_j$. The identity is obtained by noting that $P(x)$ is constant and hence $k!\frac{\dd^k}{\dd x^k} P(x)=0$. Therefore
\begin{align*}
p_{N,k+1}&=\prod_{i=0}^k \lambda_i \left[ \sum^{k}_{j=0} {\rm{e}}^{\lambda_{j}t}\prod_{n=0,n \neq j}^{k+1}  \frac{1}{(\lambda_n -\lambda_j)} + {\rm{e}}^{\lambda_{n+1}t}\prod_{n=0}^{k}  \frac{1}{(\lambda_n -\lambda_{k+1})}  \right],\\
&= \prod_{i=0}^{k} \lambda_i \left[ \sum_{j=0}^{k+1} \frac{{\rm{e}}^{\lambda_j t}}{\prod_{n \neq j} (\lambda_n -\lambda_j)} \right].
\end{align*}
Hence the statement is true for $k+1$: the result follows by induction.
\end{proof}

If there is an $i\neq j$ such that $\lambda_i=\lambda_j$ then the linear system \eqref{eq:genris} has a resonance and the explicit form of solution \eqref{eq:PNk} will be modified. We do not consider this here except to note that in the uncoupled case $r_j=r>0$ and \eqref{eq:geneigris} means there are no resonances. Assuming the Kramers rate determines the escapes, and this varies continuously in $\beta$, means that the no resonance condition is expected to apply for weak enough coupling. 

As an example, two nodes with bidirectional coupling gives
$r_0=r_1^{[0,0]} = r_2^{[0,0]}$ and $r_1 = r_1^{[0,1]} = r_2^{[1,0]}$
 so (\ref{eq:genris}) reduces to
$\dot{p}_{2,0}  = -2 r_0 p_{2.0}$,
$\dot{p}_{2,1} = 2 r_0 p_{2,0} - r_1p_{2,1}$ and
$\dot{p}_{2,2}  = r_1p_{2,1}$. The eigenvalues of the linear system are $\lambda_0=-2r_0$, $\lambda_1=-r_1$, and $\lambda_2=0$. Hence by \eqref{eq:PNk} we have
\begin{align}
p_{2,0} & =  {\rm{e}}^{\lambda_0 t}, \nonumber\\ 
p_{2,1} &=  \frac{\lambda_0}{\lambda_0-\lambda_1} \left( {\rm{e}}^{\lambda_1 t} - {\rm{e}}^{\lambda_0 t}\right),\label{eq:p202122} \\ 
p_{2,2} & =  \frac{\lambda_0 \lambda_1}{\lambda_0-\lambda_1} \left( \frac{{\rm{e}}^{\lambda_0 t}}{\lambda_0} - \frac{{\rm{e}}^{\lambda_1 t}}{\lambda_1}\right) + 1.\nonumber
\end{align}

Note that the reduction to a closed master equation with $N+1$ variables is only possible for the all-to-all coupled case where symmetry between nodes means that the order in which nodes escapes is identical on each sequence: the probability of a particular permutation $s(i)$ satisfying \eqref{eq:defines} is $1/N!$. 

\subsection{Estimation of sequential escape times}
\label{sec:estmaster}

Since $p_{N,k}(t)$ is the probability that precisely $k$ nodes have escaped, the solutions $p_{N,k}(t)$ can be used to determine the mean escape times. If we associate escape times to the $X_i$ by 
$$
\tau^{k}= \inf\{t>0~:~|\{i~:~X_i(t)=1\}|=k\},
$$
note that \eqref{eq:pNkdef} can be expressed as
$$
p_{N,k} (t)= {\mathbb{P}}\{ \tau^{k}\leq t<\tau^{k+1}\}.
$$
and
\begin{equation}
q_{N,k}(t) := {\mathbb{P}}\{\tau^{k} \leq t\} = \sum_{\ell=k}^N p_{N,\ell}.
\label{eq:mastcum}
\end{equation}
The mean time to the $k$th escape ($k>0$) is the expectation of $\tau^k$. For example, using \eqref{eq:p202122} for $N=2$ we have that $q_{2,1}= 1-p_{2,0}=1-e^{\lambda_0 t}$ and $q_{2,2}=p_{2,2}=1+(\lambda_1e^{\lambda_0t}+\lambda_0e^{\lambda_1 t})/(\lambda_0+\lambda_1)$, so that
$$
T_2^{1|0}=\int_{t=0}^{\infty} t \frac{d}{dt}[q_{2,1}(t)] \,dt = \frac{1}{|\lambda_0|}=\frac{1}{2r_0}
$$
while $T_2^{2|0}=1/|\lambda_1|+1/|\lambda_0|=1/r_1+1/(2r_0)$ for the bidirectionally-coupled two-node network considered in \sref{sec:2binet}.

The cumulative distribution \eqref{eq:cdfpt} can be approximated for any $N\geq k>\ell\geq 0$ by considering \eqref{eq:genris} with initial conditions $p_{N,\ell}(0)=1$ rather than $p_{N,0}(0)=1$. For this case we have
$$
\widehat{Q}_N^{k|\ell}(t) = \mathbb{P}\{\tau^k \leq t\}=\sum_{j\geq k} p_{N,j}.
$$
Note that
\begin{equation}
\widehat{Q}_N^{k|0}(t) = q_{N,k}(t).
\label{eq:Qk0}
\end{equation}
In the case $N=2$ we obtain 
\begin{equation}
\widehat{Q}_2^{2|1}(t)= 1-e^{\lambda_1 t}
\label{eq:Q221}
\end{equation}
and so $T_2^{2|1}=1/|\lambda_1|=1/r_1$.

The master equation approach presented above is only expected to be valid in the weak coupling regime, i.e. $\beta<\beta_{SN}$. For $\beta>\beta_{SN}$ the hypercube representation of states of the network is no longer valid; see \cite{BFG2007a}.
We use the simulations discussed in \sref{sec:2binet} for $\beta=10^{-2}$ with $\nu=0.2$, $\alpha=0.05$ and threshold $\xi=0.5$ as $ Q_N^{1|0}(t)$, $ Q_N^{2|0}(t)$ and $ Q_N^{2|1}(t)$.
We take the mean escape times $T_2^{1|0}(0.01) =133.5$ and $T_2^{2|1}(0.01)=80.94$ and compute $r_0 = 1/2T_2^{1|0}(0.01) \approx 0.00375$ and $r_1 = 1/ T_2^{2|1}(0.01) \approx 0.0124$.  We substitute these values into \eqref{eq:Qk0} for $k=1,2$ and $N=2$,  and \eqref{eq:Q221}.   \Fref{fig:simmeq} shows $ Q_2^{1|0}(t)$, $ Q_2^{2|0}(t)$ and $ Q_2^{2|1}(t)$ for $\beta=0.01$ plotted on one graph against time (linear) with $ \widehat{Q}_2^{1|0}(t)$, $ \widehat{Q}_2^{2|0}(t)$ and $ \widehat{Q}_2^{2|1}(t)$. 

\begin{figure}
\centering
\includegraphics[width=0.9\textwidth]{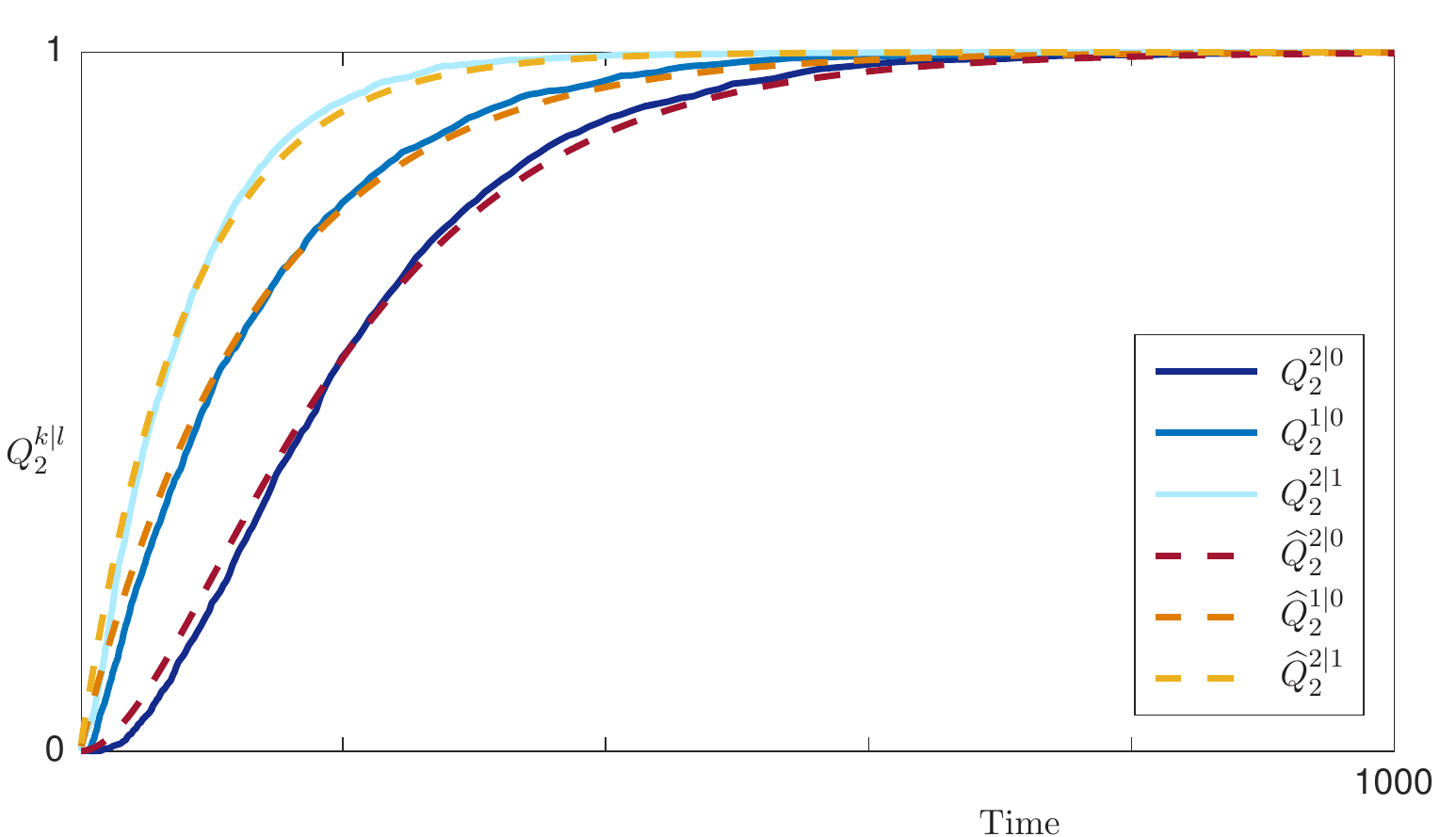}
\caption{Cumulative distributions $ Q_N^{1|0}(t)$, $ Q_N^{2|0}(t)$ and $ Q_N^{2|1}(t)$ (blues) for $\beta=0.01$, plotted with the master equation cumulative distributions $\widehat{Q}_N^{k,\ell}$ for the predicted values $r_0\approx 0.00375$ and $r_1\approx 0.0124$ (oranges).}
\label{fig:simmeq}
\end{figure}

\Fref{fig:simmeq} shows a good agreement between the simulations and the cumulative distributions obtained with the master equation approach in the weak coupling regime. For intermediate and strong coupling regimes, as illustrated for example in Figure~\ref{fig:2bibif}, the assumptions behind the master equation are no longer valid: the transition distributions may be far from well-modelled by exponential. This is beyond the scope of this paper and we leave it to future work.

\section{Discussion}
\label{sec:discuss}

In this paper we explained a number of features of the escape times discussed in \cite{benj12pheno}.
We gave improved estimates for the one node mean escape time showing the dependence on excitability $\nu$ and noise amplitude $\alpha$. We investigated sequential escapes for a system of identical bistable nodes and the dependence on coupling strength. To this end, we used the example of two bidirectionally connected nodes. We derived expressions for infinite coupling and uncoupled limits using numerical integration of the mean escape time of one node. Moreover, we gave asymptotic approximations for the first and second escape times in each of the strong, intermediate and weak coupling regimes. Here we made use of variations of the multi-dimensional Eyring-Kramers' Law~\cite{berg08EK, berg13kramers} to explain the escape times through a pitchfork bifurcation. We compared these estimates to numerical simulations. 

One of the more surprising results of this paper is that for the model system  \eqref{eq:neteq} with fixed excitability $\nu$ and noise level $\alpha$, an increase in $\beta$ has an opposite effect on the first and second mean escape times: Figure~\ref{fig:2escT} demonstrates $T_2^{1|0}$ monotonically increases for small $\beta$ while $T_2^{2|1}$ monotonically decreases. The aggregate effect is that the mean time to complete escape $T_{2}^{2|0}$ decreases and then increases. Essentially this is due to (a) strong coupling causing synchronization of the escapes but at the same time weakening the effect of noise on the system - it takes longer to escape because of the noise in the strongly coupled system is averaged. It is also (b) nonlinear effects of the interaction of bifurcations on the basin boundary in the coupled system with most likely escape paths.

In the weak coupling regime states of the network lie at the vertices of a hypercube. We used a master equation approach to find an expression for the probability of being in a given network state in the weakly all-to-all coupled case where the rate of escape from states with $k$ active nodes are equal.
We find good agreement between our numerical simulations and the distribution of escape times from a master equation model. On the other hand, the synchronization of escapes means that the master equation model breaks down above some critical coupling strength. 

As noted in \cite{act17domino}, sequential escape statistics should be of interest to characterising and modelling a wide range of processes in applications: in this work we extend these ideas to a simple case where there is bistability between an equilibrium and a limit cycle attractor. However, the simple nature of the bistable nodes we considered means that the phases effectively uncouple from the radial dynamics. For more general bistable oscillators the phase dynamics will not be so easily reduced. At least in the weak noise limit, it should be possible to develop master equation models suitable for intermediate and strong coupling. 

There are many open problems: not just generalization to more heterogeneous networks, but also non-potential systems. Eyring-Kramers' Law~\cite{berg13kramers} and the generalisation used here for non-quadratic saddles~\cite{berg08EK} require an explicit expression for the potential landscape of the system. However, the system describing two nodes with unidirectional coupling is not a potential system and the results do not hold for this case.  Analysis of escape times of non-potentail systems could be applied to, for example, energy landscapes derived from neuroimaging data~\cite{ezaki17energy, tka14search} and the bifurcations on the basin boundaries of these system could provide valuable insight into the emergent transient dynamics.

\section*{Acknowledgements}

We particularly thank the following people for their advice and perceptive suggestions: Oscar Benjamin, Chris Bick, Vadim Biktashev, Jan Sieber, John Terry, Kyle Wedgewood. We thank Robin Chapman for the proof of the identity used in Proposition~\ref{prop:master}.



\begin{thebibliography}{10}

\bibitem{AshComNic16}
{\sc P.~Ashwin, S.~Coombes, and R.~Nicks}, {\em Mathematical frameworks for
  oscillatory network dynamics in neuroscience}, The Journal of Mathematical
  Neuroscience, 6 (2016), pp.~1--92,
  \url{https://doi.org/10.1186/s13408-015-0033-6}.

\bibitem{act17domino}
{\sc P.~Ashwin, J.~Creaser, and K.~Tsaneva-Atanasova}, {\em Fast and slow
  domino effects in transient network dynamics}, apr 2017,
  \url{https://arxiv.org/abs/1701.06148}.

\bibitem{ashwin16quant}
{\sc P.~Ashwin and C.~Postlethwaite}, {\em Quantifying noisy attractors: from
  heteroclinic to excitable networks}, SIAM Journal on Applied Dynamical
  Systems, 15 (2016), pp.~1989--2016, \url{https://doi.org/10.1137/16M1061813}.

\bibitem{ben10ava}
{\sc M.~Benayoun, J.~D. Cowan, W.~van Drongelen, and E.~Wallace}, {\em
  Avalanches in a stochastic model of spiking neurons}, PLoS Comput Biol, 6
  (2010), e1000846 (13~pages),
  \url{https://doi.org/10.1371/journal.pcbi.1000846}.

\bibitem{benj12pheno}
{\sc O.~Benjamin, T.~H. Fitzgerald, P.~Ashwin, K.~Tsaneva-Atanasova,
  F.~Chowdhury, M.~P. Richardson, and J.~R. Terry}, {\em A phenomenological
  model of seizure initiation suggests network structure may explain seizure
  frequency in idiopathic generalised epilepsy}, The Journal of Mathematical
  Neuroscience, 2 (2012), pp.~1--30,
  \url{https://doi.org/10.1186/2190-8567-2-1}.

\bibitem{berg13kramers}
{\sc N.~Berglund}, {\em Kramers' law: Validity, derivations and
  generalisations}, Markov Processes and Related Fields, 19 (2013),
  pp.~459--490.

\bibitem{BFG2007a}
{\sc N.~Berglund, B.~Fernandez, and B.~Gentz}, {\em Metastability in
  interacting nonlinear stochastic differential equations: I. from weak
  coupling to synchronization}, Nonlinearity, 20 (2007), pp.~2551--2581,
  \url{https://doi.org/10.1088/0951-7715/20/11/006}.

\bibitem{BFG2007b}
{\sc N.~Berglund, B.~Fernandez, and B.~Gentz}, {\em Metastability in
  interacting nonlinear stochastic differential equations: II. large-n
  behaviour}, Nonlinearity, 20 (2007), pp.~2583--2614,
  \url{https://doi.org/10.1088/0951-7715/20/11/007}.

\bibitem{berg08EK}
{\sc N.~Berglund and B.~Gentz}, {\em The Eyring-Kramers law for potentials with
  nonquadratic saddles}, Markov Processes and Related Fields, 16 (2010),
  pp.~549--598.

\bibitem{bovier04meta}
{\sc A.~Bovier, M.~Eckhoff, V.~Gayrard, and M.~Klein}, {\em Metastability in
  reversible diffusion processes: I. Sharp asymptotics for capacities and exit
  times}, Journal of the European Mathematical Society, 6 (2004), pp.~399--424,
  \url{https://doi.org/10.4171/JEMS/14}.

\bibitem{daffer98}
{\sc A.~Daffertshofer}, {\em Effects of noise on the phase dynamics of
  nonlinear oscillators}, Physical Review E, 58 (1998), pp.~327--338,
  \url{https://doi.org/10.1103/PhysRevE.58.327}.

\bibitem{AutoOrig}
{\sc E.~J. Doedel}, {\em Auto: A program for the automatic bifurcation analysis
  of autonomous systems}, Congr. Numer, 30 (1981), pp.~265--284.

\bibitem{doedel07lec}
{\sc E.~J. Doedel}, {\em {Lecture notes on numerical analysis of nonlinear
  equations}}, in Numerical Continuation Methods for Dynamical Systems, B.
  Krauskopf, H. M. Osinga, and J. Gal\'an-Vioque, eds. Springer, 2007,
  pp.~1--49.

\bibitem{auto}
{\sc E.~J. Doedel, R.~C. Paffenroth, A.~R. Champneys, T.~F. Fairgrieve, Y.~A.
  Kuznetsov, B.~E. Oldeman, B.~Sandstede, and X.~J. Wang}, {\em {AUTO-07P:
  Continuation and bifurcation software for ordinary differential equations}},
  2007, \url{http://indy.cs.concordia.ca/auto} (accessed 2016/03/01).
\newblock Version 8.0.

\bibitem{eyring35activ}
{\sc H.~Eyring}, {\em The activated complex in chemical reactions}, The Journal
  of Chemical Physics, 3 (1935), pp.~107--115,
  \url{https://doi.org/10.1063/1.1749604}.

\bibitem{ezaki17energy}
{\sc T.~Ezaki, T.~Watanabe, M.~Ohzeki, and N.~Masuda}, {\em Energy landscape
  analysis of neuroimaging data}, nov 2016, \url{https://arxiv.org/abs/1611.05137}.

\bibitem{gard83hand}
{\sc C.~W. Gardiner}, {\em Handbook of stochastic methods for physics,
  chemistry and the natural sciences}, Springer-Verlag, Berlin Heidelberg,
  1983.

\bibitem{kalitzin10}
{\sc S.~N. Kalitzin, D.~N. Velis, and F.~H.~L. da~Silva}, {\em
  Stimulation-based anticipation and control of state transitions in the
  epileptic brain}, Epilepsy \& Behavior, 17 (2010), pp.~310--323,
  \url{https://doi.org/10.1016/j.yebeh.2009.12.023}.

\bibitem{kloeden03num}
{\sc P.~E. Kloeden, E.~Platen, and H.~Schurz}, {\em Numerical solution of SDE
  through computer experiments}, Springer Science \& Business Media, 2012.

\bibitem{kramers40brown}
{\sc H.~A. Kramers}, {\em Brownian motion in a field of force and the diffusion
  model of chemical reactions}, Physica, 7 (1940), pp.~284--304,
  \url{https://doi.org/10.1016/S0031-8914(40)90098-2}.

\bibitem{kraus07comp}
{\sc B.~Krauskopf and H.~M. Osinga}, {\em {Computing invariant manifolds via
  the continuation of orbit segments}}, in Numerical Continuation Methods for
  Dynamical Systems, B. Krauskopf, H. M. Osinga, and J. Gal\'an-Vioque, eds.
  Springer, 2007, pp.~117--154.

\bibitem{rabinovich08}
{\sc M.~Rabinovich, R.~Huerta, and G.~Laurent}, {\em Transient dynamics for
  neural processing}, Science, 321 (2008), pp.~48--50,
  \url{https://doi.org/10.1126/science.1155564}.

\bibitem{rabinovich11}
{\sc M.~I. Rabinovich and P.~Varona}, {\em Robust transient dynamics and brain
  functions}, Frontiers in computational neuroscience, 5 (2011), 24 (10~pages),
  \url{https://doi.org/10.3389/fncom.2011.00024}.

\bibitem{tka14search}
{\sc G.~Tka{\v{c}}ik, O.~Marre, D.~Amodei, E.~Schneidman, W.~Bialek, and M.~J.
  Berry~II}, {\em Searching for collective behavior in a large network of
  sensory neurons}, PLoS Comput Biol, 10 (2014), e1003408 (23~pages),
  \url{https://doi.org/10.1371/journal.pcbi.1003408}.

\bibitem{wendling16}
{\sc F.~Wendling, P.~Benquet, F.~Bartolomei, and V.~Jirsa}, {\em Computational
  models of epileptiform activity}, Journal of neuroscience methods, 260
  (2016), pp.~233--251, \url{https://doi.org/10.1016/j.jneumeth.2015.03.027}.

\end{thebibliography}
\end{document}